\documentclass[moor]{informs4}

\usepackage[T1]{fontenc}

\usepackage{eqndefns-left} % For checking the display equation width and equation environment definitions %

\RequirePackage{tgtermes}
\RequirePackage{newtxtext}
\RequirePackage{newtxmath}
\RequirePackage{bm}
\RequirePackage{endnotes}
\OneAndAHalfSpacedXII % Current default line spacing !!!!

\usepackage{algorithm}
\usepackage{algpseudocode}
\usepackage{tikz}
\usepackage{enumerate}
\usepackage{bbm}
\usepackage{hhline}
\usepackage{multirow}

\usepackage[sort&compress]{natbib}
 \bibpunct[, ]{[}{]}{,}{n}{}{,}%
 \def\bibfont{\small}%
 \def\newblock{\ }%
\EquationsNumberedThrough    % Default: (1), (2), ...
\TheoremsNumberedThrough     % Preferred (Theorem 1, Lemma 1, Theorem 2)
\ECRepeatTheorems  %  

\MANUSCRIPTNO{MOOR-0001-2024.00}

\begin{document}
%%%%%%%%%%%%%%%%

%\setcounter{page}{1}
%\DOI{10.1287/xxxx.0000.0000}%

\RUNAUTHOR{Chorobura, Necoara and Pesquet}

% \RUNTITLE{Predictive Maintenance in Manufacturing}
\RUNTITLE{An adaptive forward-backward-forward splitting algorithm for solving pseudo-monotone inclusions }

% \TITLE{Optimal Resource Allocation in Humanitarian Logistics: A Stochastic Programming Approach}
\TITLE{An adaptive forward-backward-forward splitting algorithm for solving pseudo-monotone inclusions }

\ARTICLEAUTHORS{%
%\AUTHOR{John Doe,\textsuperscript{a} Jane Smith,\textsuperscript{b}}
%\AFF{\textsuperscript{a}Department of Industrial Engineering, University of XYZ, \EMAIL{john.doe@xyz.edu; \textsuperscript{b}Department of Computer Science, University of ABC, \EMAIL{jane.smith@abc.edu}} 
\AUTHOR{Flavia Chorobura}
\AFF{Automatic Control and Systems
	Engineering Department, National University of Science and Technology  Politehnica Bucharest, 060042
	Bucharest, Romania and Federal Institute of Education, Science and Technology of Paraná, campus Foz do Iguaçu, Brazil, \EMAIL{flavia.chorobura@ifpr.edu.br}}

\AUTHOR{Ion Necoara}
\AFF{Automatic Control and Systems
	Engineering Department, National University of Science and Technology Politehnica Bucharest, 060042
	Bucharest  and  Gheorghe Mihoc-Caius Iacob  Institute of Mathematical Statistics and Applied Mathematics of the Romanian Academy, 050711 Bucharest, Romania,   \EMAIL{ion.necoara@upb.ro}}

\AUTHOR{Jean-Christophe Pesquet}
\AFF{University Paris-Saclay, CentraleSup\'elec, CVN, Inria, Gif-sur-Yvette, France, \EMAIL{jean-christophe.pesquet@centralesupelec.fr}}
} % end of the block

\ABSTRACT{%
% Enter your abstract
In this paper, we propose an adaptive forward-backward-forward splitting algorithm for finding a zero of a pseudo-monotone operator which is split as a sum of three operators: the first  is continuous single-valued, the second  is Lipschitzian, and the third is maximally monotone. This setting  covers, in particular,  constrained minimization scenarios, such as  problems having  smooth and convex functional constraints (e.g., quadratically constrained quadratic programs) or problems with a pseudo-convex objective function minimized over a simple closed convex  set (e.g., quadratic over linear fractional programs).  For the general problem, we design a forward-backward-forward splitting type method based on  novel adaptive stepsize strategies. Under an additional generalized Lipschitz property of the first operator,  sublinear convergence rate is derived for the sequence generated by our adaptive algorithm. Moreover, if the sum is uniformly pseudo-monotone, linear/sublinear rates are derived depending on the parameter of uniform pseudo-monotonicity. 
Preliminary numerical experiments  demonstrate the good performance of our method when compared to some existing optimization methods and software.
}%

\FUNDING{The research leading to these results has received funding from: ITN-ETN project TraDE-OPT funded by the European Union’s Horizon 2020 Research and Innovation Programme under the Marie Skolodowska-Curie grant agreement No. 861137; UEFISCDI, Romania, PN-III-P4-PCE-2021-0720, under project L2O-MOC, nr. 70/2022.}

%Supplemental Material:
%Data Ethics & Reproducibility Note:

%\KEYWORDS{Stochastic programming, Decision support,Uncertainty, Disaster response, Optimization}

%\HISTORY{Received: Month DD, YYYY; Accepted: Month DD, YYYY; Published Online: Month DD, YYYY}

\maketitle
%%%%%%%%%%%%%%%%%%%%%%%%%%%%%%%%%%%%%%%%%%%%%%%%%%%%%%%%%%%%%%%%%%%%%%

% Text of your paper here

\section{Introduction}\label{sec:Intro}
\noindent Let $\mathbb{H}$ be a finite-dimensional real vector space  endowed with a scalar product $\langle \cdot, \cdot\rangle$ and the corresponding norm  $\|\cdot\|$. Our goal is to find a zero of a sum of three operators $A\colon \mathbb{H}\to \mathbb{H}$, $B\colon \mathbb{H}\to \mathbb{H}$, and $C\colon \mathbb{H}\ \to 2^{\mathbb{H}}$, that is
 \vspace{-0.3cm}	
 \begin{align}
    \label{prob}
    \text{find}\;\bar{z} \in \mathbb{H}\;\text{such that}\;\;
		0 \in A\bar{z} + B\bar{z} + C\bar{z},
   \vspace{-0.7cm}
	\end{align}

\vspace{-0.5cm}
\noindent where  $A+B+C$ is pseudo-monotone, and $A$ satisfies some generalized smoothness condition,  $B$ is smooth, and $C$ is maximally monotone  as detailed in Assumption \ref{ass1} of Section \ref{se:assum}. Finding a zero of a sum of operators is a very general problem and covers, in particular, constrained optimization, and  minimax optimization problems frequently encountered in  signal processing \cite{HuaPal:14},  triangulation in computer vision \cite{AhoAgaTho:12}, semi-supervised learning \cite{CheLi:21}, learning of kernel matrices \cite{Lan:02},  steering direction estimation for RADAR detection \cite{Mai:11}, generative adversarial networks \cite{MerPapPil:18}
among others. 

\medskip 

\noindent \textit{Previous work.} The problem of finding a zero of a sum of operators is considered in many works. For example, \cite{BriDav:18,BuiCom:22,Com:24,ComPes:12,DavWot:17,JiaVan:22,Tse:00} cover the monotone case, while \cite{AlaKimWri:24,CaiZhe:22,Pet:22} consider the nonmonotone case. In \cite{DavWot:17,JiaVan:22} all  three operators are assumed maximally monotone and, additionally,  the first is Lipschitz continuous. Under these settings,  algorithms based on resolvent
and forward operators, activated one at a time successively, are proposed together with a detailed convergence analysis. Furthermore, finding a zero of a sum of two maximally monotone operators,  $A+C$,  such that $A$ is a continuous single-valued operator, is investigated in \cite{Tse:00} and a forward-backward-forward algorithm is proposed (also known as Tseng's algorithm),  where the stepsize is chosen constant when $A$ is Lipschitz or based on an Armijo-Goldstein-type rule, otherwise. Linear rate was derived for this method when $A+C$ is strongly monotone. In \cite{AlaKimWri:24,CaiZhe:22,Pet:22},  $A$ is assumed Lipschitz continuous, possible nonmonotone, and $C$ is maximally monotone, such that either $A+C$ satisfies the weak Minty condition or  a cohypomonotonicity assumption.  In particular, \cite{Pet:22} considers an extragradient algorithm  with adaptive and constant stepsizes, which reduces, for a specific choice of stepsize, to the forward-backward-forward algorithm in the monotone case. Moreover, \cite{CaiZhe:22} analyzes an optimistic gradient algorithm,  while in \cite{AlaKimWri:24} algorithms  based on classical Halpern and  Krasnosel’skii-Mann iterations are analyzed. For all these methods, under suitable assumptions, sublinear rates are derived. Finally, finding a zero of a sum of three operators $A+B+C$  is considered in \cite{ComPes:12}, where $A, C$ are maximally monotone and $B$ is Lipschitz continuous and monotone, and asymptotic convergence is proved  for an error-tolerant forward-backward-forward~algorithm. 

\vspace{0.1cm}

\noindent The forward-backward-forward algorithm was also extended to solve variational inequalities. For example, \cite{BotCseVuo:20,Ton:23} consider a variational inequality, where the operator is Lipschitz continuous, and a (modified) Tseng algorithm is employed with a constant stepsize or an adaptive stepsize, so that it is not necessary to know the Lipschitz constant. Convergence is derived when the operator is pseudo-monotone. Moreover, the Lipschitz assumption on the operator involved in the variational inequality is relaxed in \cite{ThoVuo:19}, the operator being assumed continuous. Then, Tseng's algorithm is considered with an Armijo-Goldestein rule for the stepsize. Under standard conditions, the weak and strong convergence of the  method is obtained in the pseudo-monotone case. Our approach differs from \cite{ThoVuo:19}, as we consider that the operator $A$ satisfies a relaxed Lipschitz condition and we employ Tseng's algorithm with  novel adaptive  stepsize rules (e.g., based on the positive root of a polynomial equation). Others methods  for solving variational inequalities with a Lipschitz operator in the monotone case were considered e.g., in  \cite{Nem:04} and in the nonmonotone case (under weak Minty condition)~in~\cite{DiaDasJor:21}. 

\vspace{0.1cm}

\noindent Furthermore, specific algorithms were also developed for particular classes of   variational inequalities, such as convex-concave  minimax optimization problems  \cite{ChaPoc:11,Con:13}. More specifically, these papers address problems of the form: 
 \vspace{-0.1cm}
\begin{equation}
\label{prob:3}
\min_{x \in \mathcal{X}} \max_{y\in \mathcal{Y}} \; \langle Lx,y \rangle + \varphi(x) - \psi(y), 
 \vspace{-0.1cm}
\end{equation} 
where $\mathcal{X}$ and $\mathcal{Y}$ are Hilbert spaces, $L$ is a linear operator and 
$\varphi: \mathcal{X} \to \bar{\mathbb{R}} := \mathbb{R} \cup \{+\infty\}$ and $\psi: \mathcal{Y} \to \bar{\mathbb{R}} $ are proper, convex, lower semicontinuous functions. For such problems, a primal-dual proximal algorithm is proposed in \cite{ChaPoc:11}  for which a sublinear rate is derived in the optimality measure:
 \vspace{-0.3cm}
\begin{align*}
 G(\bar{x},\bar{y}) = 
\max_{y \in \mathcal{Y}} \langle L\bar{x}, y \rangle - \psi(y) + \varphi(\bar{x}) - \min_{x\in \mathcal{X}} \langle Lx, \bar{y} \rangle + \varphi(x) -\psi(\bar{y}),
 \vspace{-0.3cm}
\end{align*}

\vspace{-0.15cm}

\noindent for a given $(\bar{x},\bar{y})\in \mathcal{X} \times \mathcal{Y}$. An extension of the algorithm from \cite{ChaPoc:11} is given  in \cite{Con:13}, where $\varphi$ is split as  $\varphi_1 + \varphi_2$, with $\varphi_{1}: \mathcal{X} \to \mathbb{R}$ convex, differentiable with a  Lipschitz continuous gradient, while $\varphi_{2}$ is a  proper, convex, lower semicontinuous function.  It is proved that this algorithm converges weakly to a solution to problem \eqref{prob:3} and,  if  $\varphi_1=0$, then  \cite{Con:13} recovers the primal–dual algorithm~in~\cite{ChaPoc:11}.

\medskip
 
\noindent \textit{Contributions.} In this paper, we propose a method for finding a zero of a sum of three operators, which are not necessarily monotone.  For this general problem we design a forward-backward-forward splitting type method based on  novel adaptive stepsize strategies and then perform a detailed convergence analysis. More specifically, our main contributions are the following.

\vspace{0.1cm}

\noindent (i) \textit{Mathematical modelling:} We consider the general problem \eqref{prob} of finding a zero of a sum of three operators, $A+B+C$, such that $A$ is continuous, $B$ is Lipschitz continuous, and $C$ is maximally monotone. In contrast to other works that assume $A$ to be Lipschitz continuous and the sum to be monotone, we relax these conditions, i.e., we require  the operator $A$ to satisfy a generalized Lipschitz condition and the sum to be pseudo-monotone. Our assumptions cover important classes of optimization problems such as  problems minimizing  smooth and convex functional constraints (e.g., quadratically constrained quadratic programs) or problems minimizing pseudo-convex objective functions over a simple closed convex  set (e.g., quadratic over linear fractional programs).

\vspace{0.1cm}

\noindent (ii) \textit{Algorithm:} For solving this general problem we propose  a variant of  the forward-backward-forward algorithm \cite{Tse:00}, based on \textit{two novel adaptive stepsize strategies}.  In contrast to previous works where computationally expensive Armijo-Goldestein stepsize rules are used when the operator is continuous, we propose two \textit{adaptive} stepsize strategies that require finding the root of a certain nonlinear equation whose coefficients depend on the current iterate and on the parameters characterizing the operator properties. In particular, for quadratically constrained quadratic (resp. quadratic over linear fractional) programs the stepsize is computed solving a second-order (resp. third-order) polynomial equation.  

\vspace{0.1cm}

\noindent (iii) \textit{Convergence analysis:} Within the considered settings, we provide a detailed convergence analysis for the  forward-backward-forward algorithm based on our adaptive stepsize rules.    In particular, when the sum of the operators is  pseudo-monotone, we prove the global  asymptotic convergence for the whole sequence generated by the algorithm and,  additionally, establish sublinear convergence rate.  An improved linear rate is obtained when the sum  is uniformly pseudo-monotone of order  $q\in [1,2]$. 

\vspace{0.1cm}

\noindent (iv) \textit{Experiments:} Detailed numerical experiments using synthetic and real data  demonstrate the effectiveness of our method and allows us to evaluate its performance when compared to some existing state-of-the-art optimization methods from \cite{Tse:00,ThoVuo:19},  and existing software \cite{Gurobi}.\\
In conclusion, enhancing a forward-backward-forward splitting algorithm by introducing novel adaptive stepsize strategies and considering suitable assumptions on the generalized problem (covering a wide spectrum of applications) to allow the derivation of a complete convergence analysis, represent important contributions for solving inclusion problems.
\vspace{-0.2cm}

\section{Background}
\noindent We denote by $\text{zer}(A)$ the set of zeros of an operator $A$,  by  $\Gamma_{0}(\mathbb{H})$  the set of proper lower semicontinuous convex functions on $\mathbb{H}$ with values in $]-\infty,+\infty]$, by $\mathbf{0}_{m\times m}$ the $m\times m$ null matrix and by $\mathbf{0}_{m}$ the $m$-dimensional null vector. Further, let us recall the definition of the subdifferential of a convex function.
    \begin{definition}
        The subdifferential of a proper convex function $f: \mathbb{H} \to \bar{\mathbb{R}}$ is the set-valued operator $\partial f: \mathbb{H}  \to 2^{\mathbb{H}}$ which maps every point $x\in \mathbb{H} $ to the set
        \vspace{-0.3cm}
        \begin{align*}
            \partial f(x) = \{u \in \mathbb{H}  \,|\, (\forall y \in \mathbb{H} )  \,\  \langle y-x, u \rangle + f(x) \leq f(y)  \}.
            \vspace{-0.5cm}
        \end{align*}
    \end{definition}

    \vspace{-0.5cm}
    \noindent Note that $\partial f(x) = \varnothing$ for $x \not \in \text{dom}f$.  For example, let $D$ be a nonempty closed and convex subset of $\mathbb{H}$ and let its indicator function $\iota_{D}$ be defined as
    \begin{equation}\label{eq:indicator}
        \iota_{D}: \mathbb{H} \to \bar{\mathbb{R}}: \;\;  x \mapsto 
        \begin{cases}
        0  & \text{if} \; x\in D\\ 			 + \infty, &\; \text{otherwise.}
        \end{cases}
    \end{equation}
    \noindent Then,  $\partial \iota_{D} = \mathcal{N}_{D}$, where $\mathcal{N}_{D}$ is the normal cone  to $D$, i.e.
    \begin{equation}\label{eq:normalcone}
        \mathcal{N}_{D}(x) = 
        \begin{cases}
        \left\{ u \in \mathbb{H} \; | \;
        (\forall y \in D)\;\;
        \langle y-x, u \rangle \leq 0\right\} &  \text{if} \; x \in D    \\
        \varnothing & \text{otherwise.} 
        \end{cases}   
    \end{equation}

\noindent Moreover, if $f$ is differentiable at a point $x \in \text{dom} f$, its gradient is denoted by $\nabla f(x)$.  Let us also recall the definition of functions with  H\"older continuous gradient.
    \begin{definition}
    \label{def:holder}
    Let $\nu \in ]0,1]$. Then, the  differentiable function $g\colon \mathbb{H} \to \mathbb{R}$ has a $\nu$-H\"older continuous gradient,  if there exists  $L_g > 0$ such that
    \begin{equation}
        \label{eq:holder}
        (\forall (x,w) \in \mathbb{H}^2)\quad 
        \|\nabla g(x) - \nabla g(w) \| \leq L_g \|x-w\|^{\nu}.
    \end{equation}
\end{definition}
\noindent If $g$ has $\nu$-H\"older continuous gradient, then the following inequality holds, see \cite[Lemma 1]{Yas:16}: 
\vspace{-0.8cm}
\begin{align}
\label{eq:holderineq}
(\forall (x,w)\in \mathbb{H}^2)\quad 
|g(w) - g(x) - \langle \nabla g(x), w-x\rangle| \leq \dfrac{L_{g}}{1+\nu} \|w-x\|^{1+\nu}.
\vspace{-0.8cm}
\end{align}

\vspace{-0.5cm}
\noindent Next, we present the definitions of pseudo-convex functions and   operators.

\begin{definition}
Let  $\mathcal{X} \subseteq \mathbb{H}$ be an open set,  let $f \colon \mathcal{X} \to \mathbb{R}$ be a differentiable function and let $\mathcal{Z}$ be a subset of $\mathcal{X}$. Then, $f$ is said to be pseudo-convex on ${\mathcal{Z}}$ if, for every $(x,w)\in \mathcal{Z}^2$, one has:
     \vspace{-0.2cm}
    \begin{equation*}
     \langle \nabla f (x), w - x \rangle \geq 0 \,\ \Longrightarrow  \,\ f(x) \leq f(w).
      \vspace{-0.2cm}
    \end{equation*}
\end{definition}

\noindent Clearly, any convex function is pseudo-convex and any stationary point of a pseudo-convex function is a global minimum. However, there are also pseudo-convex functions that are not convex. For example,   consider an open convex set $\mathcal{X} \subset \mathbb{R}^{n}$ and differentiable functions $g: \mathcal{X} \to [0,+\infty[$ and $h\colon \mathcal{X} \to ]0,+\infty[$ such that $g$ is convex and $h$ is concave. Then, the function $f \colon \mathcal{X}\to ]0,+\infty[\colon x \mapsto g(x)/h(x)$, is pseudo-convex on any subset of $\mathcal{X}$ \cite{BorLew:06}. Other examples of pseudo-convex functions are given in Example \ref{ex:frac} below, see also \cite{Man:65}. The notion of pseudo-convexity was also extended to nondifferentiable functions, see for example \cite{Aus:98}.  

\begin{definition}
\label{def:pseudo-mon}
An operator $T\colon \mathbb{H}\to 2^{\mathbb{H}}$ is said to be pseudo-monotone if     
 \vspace{-0.5cm}
\begin{align*}
    (\forall (x,y) \in \mathbb{H}^2)\;
    (\exists \hat{x} \in Tx)\;\; \langle \hat x, y-x \rangle \geq 0 \quad \Longrightarrow \quad
(\forall  \hat{y} \in Ty)\;\;
\langle \hat y, y-x \rangle \geq 0.
 \vspace{-0.7cm}
    \end{align*}
    \end{definition}   
    \vspace{-0.5cm}
 For example, \cite{KarSch:90} shows  that any differentiable pseudo-convex function has a pseudo-monotone gradient. In addition, \cite{Aus:98} proves that a lower semicontinuous radially continuous function is pseudo-convex if and only if its    subdifferential is pseudo-monotone. Moreover, note that every monotone operator is pseudo-monotone

\medskip 

\noindent Finally, let us recall the definition of the resolvent of an operator $C: \mathbb{H} \to 2^{\mathbb{H}}$. The resolvent of $C$ is the operator $J_{C} = (\operatorname{Id} + C)^{-1}$, that is
 \vspace{-0.1cm}
   \begin{equation*}
       (\forall(x,p) \in  \mathbb{H}^2) \quad p \in J_{C}x \;\; \iff \;\; x-p \in Cp.
    \vspace{-0.1cm}
   \end{equation*}
 % \noindent Recall that, 
\noindent If $C: \mathbb{H} \to 2^{\mathbb{H}}$ is maximally monotone, then $J_{C}$ is single-valued, defined everywhere on $\mathbb{H}$, and firmly nonexpansive \cite{BauCom:11}. Moreover, if $C = \partial f$, the subdifferential operator of a convex function $f$, then its resolvent is the proximal mapping $\text{prox}_{f}$. If $f = \iota_{D}$, where $\iota_{D}$ is defined in \eqref{eq:indicator} and $D$ is a nonempty closed convex subset of $\mathbb{H}$, then  $\text{prox}_{\gamma \iota_{D}} = \text{proj}_{D}$, where $\text{proj}_{D}$ is the projection operator onto  the set $D$.

%%%%%%%%%%%%%%%%%%%%%%%%%%%%%%%%%%%

\section{Assumptions and examples}
\label{se:assum}
In this section we provide  our main assumptions and also several examples of problems that fit into our framework. First, let us 
present  our standing assumptions for  the operators $A$, $B$,  and $C$.
	\begin{assumption}
		\label{ass1}
	The following assumptions hold for problem \eqref{prob}:
		\begin{enumerate}[i)]
              \item\label{ass1iii} $C$ has nonempty closed  convex domain, $\operatorname{dom} C$, and  is maximally monotone.
              %on this domain.
            \item\label{ass1i} A is a continuous single-valued operator on $\operatorname{dom} C \subseteq \mathbb{H}$.
            \item\label{ass1ii} $B$ is  a single-valued operator and Lipschitz on  $\operatorname{dom}C$  with a Lipschitz constant $L_B>0$ (when $B = 0$, we can take an arbitrarily small positive value for $L_B$). 
			\item\label{ass1iv} $A+B+C$ is a pseudo-monotone operator. 
   %on $\operatorname{dom} C$. 
			\item\label{ass1v} There exist $(\zeta,\tau) \in ]0,+\infty[^2$
             such that for every $(u,w) \in \mathbb{H}^2$, $\gamma \in ]0,+\infty[$,  $q = \operatorname{proj}_{\operatorname{dom} C}w$, and $z= q - \gamma u$, the following holds:
			\vspace*{-0.5cm}
            \begin{align}
				\label{ass:resC}
            \|q - J_{\gamma C} z\| \leq \gamma(\zeta \|u\| + \tau).
             \vspace*{-1.5cm}
			\end{align}  
            
            \vspace*{-0.5cm}
			\item\label{ass1vi} $A$ satisfies a generalized Lipschitz condition, that is, there exist $\mu  \in ]0,2]$, $(\beta,\theta) \in [2,+\infty[^2$ and continuous functions $a$, $b$, and $c$ from $\mathbb{H}$  to $[0,+\infty[$ such that,  for every $(z_{1},z_2)\in (\operatorname{dom} C)^2$, 
   %we have
             \begin{equation}
            \label{eq:operatorA}
	           \|Az_{1} - Az_{2}\|^2 \leq a(z_{1}) \|z_{1}-z_{2}\|^{\mu} + b(z_{1}) \|z_{1}-z_{2}\|^{\theta} + c(z_{1}) \|z_{1}-z_{2}\|^{\beta}.
            \end{equation}
		\end{enumerate}
	\end{assumption}

\begin{remark}
Assumption \ref{ass1}.\ref{ass1vi} can be generalized to more than three terms, i.e., $A$ satisfies a generalized Lipschitz condition if there exist $\mu  \in ]0,+\infty[$, $(\theta_{i})_{1\le i \le m} \in [2,+\infty[^m$, and continuous functions $a$ and $(b_{i})_{1\le i \le m}$ from $\mathbb{H}$  to $[0,+\infty[$ such that 
             \begin{equation*}
%          \label{eq:operatorA}
            (\forall (z_{1},z_2)\in (\operatorname{dom} C)^2)\quad
	           \|Az_{1} - Az_{2}\|^2 \leq a(z_{1}) \|z_{1}-z_{2}\|^{\mu} + \sum_{i=1}^{m} b_{i}(z_{1}) \|z_{1}-z_{2}\|^{\theta_{i}}.
            \end{equation*}
Although the convergence analysis from this paper can be derived under this more general  condition,  Assumption \ref{ass1}.\ref{ass1vi} turns out to be sufficient in 
most of the applications of interest. Additionally, note that we require $a(\cdot), b(\cdot)$ and $c(\cdot)$ to depend only on $z_1$ in order to make our adaptive stepsize choices introduced in the following sections dependent only on the current  iterate.
\end{remark}

\noindent Note  that our assumptions are quite general.  Next, we present some important examples of optimization problems that can be recast as problem ~\eqref{prob}, showing the versatility of our settings. 

\begin{example}
\label{ex:3terms}
(Minimizing the sum of three functions). The most straightforward example of inclusion \eqref{prob} arises from the optimization problem:
\vspace{-1cm}
 
\begin{align} 
\label{eq:3terms}
\min_{x \in \mathbb{R}^{n}}  F(x) \; :=f(x) + g(x) + h_1(x) +h_2(Lx), 
\end{align}

\vspace{-0.5cm}  

\noindent where $L\in \mathbb{R}^{m\times n}$, $f\in \Gamma_{0} (\mathbb{R}^n)$ has a $\nu$-H\"older continuous gradient with constant $L_{f}>0$, $g\in \Gamma_{0} (\mathbb{R}^n)$  has a Lipschitz continuous gradient with constant $L_g > 0$, and  $h_2 \in \Gamma_{0} (\mathbb{R}^m)$ is finite at a point in the relative interior of the range of $L$.  Moreover, we assume that  $h_{1}$ is the indicator function of a nonempty closed convex set $D$ and $h_2$ is Lipschitz on its domain  with modulus $L_{h_2}>0$. We assume that $L(D) \subseteq \operatorname{dom} h_{2}$ and there is a point in the intersection of the relative interiors of $L(D)$ and $\operatorname{dom}h_{2}$. The latter condition ensures that $\partial (h_{1} + h_{2} \circ L)  = \partial h_{1} + L^\top \circ \partial h_2 \circ L$. This formulation covers composite ($f=0$), H\"older composite ($g=0$),  or hybrid   composite problems, respectively. The first-order optimality condition for \eqref{eq:3terms} at $\bar x$ reduces to $0 \in A\bar x + B \bar x + C \bar x$, where  $A=\nabla f$, $B=\nabla g$, and $C=\partial h_1 + L^\top \circ \partial h_2 \circ L$. Next, we show that Assumption \ref{ass1} holds for this example:

\noindent (i) - (iii) Note that, since $h_1 \in \Gamma_{0} (\mathbb{R}^n)$, $h_2 \in \Gamma_{0} (\mathbb{R}^m)$, the gradient of $f$ is continuous and g has Lipschitz continuous gradient, operator $C$ is maximally monotone, operator $A$ is a single-valued continuous operator, and operator $B$ is Lipschitz continuous with constant $L_{B} = L_g$.

\noindent (iv) Since the functions $f$ and $g$ are convex, then operators $A$ and $B$ are monotone. Moreover, since  $C$ is maximal monotone, we have that the operator $A+B+C$ is monotone, and consequently, pseudo-monotone. 

\noindent (v)  
 Let $(u,w)\in \mathbb{H}^2$,  $q = \text{proj}_{\text{dom} C}w$, and 
$z= q - \gamma u$. Then, we  have
         \vspace{-0.5cm}
        \begin{align*}
            \operatorname{dom}C = D \quad \text{and}\quad J_{\gamma C} =  \operatorname{prox}_{\gamma(h_{1}+h_{2}\circ L)}.
         \vspace{-0.2cm}
        \end{align*}

\vspace{-0.5cm}
\noindent In \cite{AdlBouCau:19} it was proved that
         \vspace{-0.5cm}
        \begin{align*}
            J_{\gamma C} = \text{proj}_{\text{dom}C }\circ \text{prox}^{h_{1}}_{\gamma h_{2} \circ L}, \; \text{where} \;
            \text{prox}^{h_{1}}_{\gamma h_{2} \circ L} = (\text{Id} + \gamma L^{\top} \circ  \partial h_{2} \circ L  \circ \text{proj}_{\text{dom} C})^{-1}.
             \vspace{-0.1cm}
        \end{align*}     

\vspace{-0.5cm}
\noindent Define $p=\text{prox}^{h_{1}}_{\gamma h_{2} \circ L}(z)$. Then, $p + \gamma \hat{p}=z$ for some $\hat{p} \in L^\top \circ \partial h_{2} \circ L  (\text{proj}_{\text{dom} C} p)$. Moreover, $\|\hat{p}\| \leq L_{h_2} \|L\|$  and 
         \vspace{-0.5cm}
        \begin{align*}
           &  \|\text{proj}_{\text{dom} C}w  - J_{\gamma C} z\| \leq \|\text{proj}_{\text{dom} C} w -  \text{prox}^{h_{1}}_{\gamma h_{2} \circ L}(z)\| = \|\text{proj}_{\text{dom} C} w - z + \gamma \hat{p}\| \\ &\leq \|\text{proj}_{\text{dom} C} w - z \| + \gamma\|\hat{p}\|  \leq \|\text{proj}_{\text{dom} C} w-z\| + \gamma L_{h_2} \|L\| = \gamma\|u\| + \gamma L_{h_{2}} \|L\|,
            \vspace{-0.1cm}
        \end{align*}

        \vspace{-0.5cm}
        \noindent where, in the first inequality, we have used the nonexpansiveness of the projection operator and $\text{proj}_{\text{dom} C}(\text{proj}_{\text{dom} C}(w)) = \text{proj}_{\text{dom} C}(w)$, and in the last one, we  have used the linear relation 
        between $z$, $\text{proj}_{\text{dom} C} w$, and $u$. Therefore, in this case,  $\zeta=1$ and $\tau = L_{h_{2}} \|L\|$.

\noindent (vi) From \eqref{eq:holder} and the definition of operator $A$, we have, for every $(z_{1},z_2)\in D^2$, 
\begin{equation*}
    \|Az_{1} - Az_{2}\|^2 \leq L_{f}^2\|z_{1}-z_{2}\|^{2\nu}
\end{equation*}
Hence, $a(z_1) = L_{f}^2$, $\mu=2\nu$, $b(z_{1})=0$, and $c(z_{1})=0$.

\end{example}    
    
\begin{example}(Minimax problems).  
\label{ex:minimax}
Let $m$ and $n$ be positive integers, and consider the following minimax problem:  
         \vspace{-0.1cm}
		\begin{equation}
			\min_{x \in \mathbb{R}^n} \max_{ y \in \mathbb{R}^m} 	F(x,y) + \varphi(x) - \psi(y), \label{eq:prob}
         \vspace{-0.1cm}
		\end{equation}
		where $F: \mathbb{R}^n \times \mathbb{R}^{m} \to \mathbb{R}$ is a differentiable function, $\psi = \psi_1 + \psi_2$,  and $\varphi = \varphi_1 + \varphi_2$, with $\psi_1$ and $\varphi_1$ having Lipschitz gradients, $\varphi_2 \in \Gamma_{0} (\mathbb{R}^n)$, and $ \psi_2 \in \Gamma_{0} (\mathbb{R}^{m})$. Note that the minimax problem \eqref{eq:prob} is more general than  problem \eqref{prob:3} considered in previous works, as we allow more general expressions for $F$ beyond bilinear terms.  The first-order optimality conditions for this problem are equivalent to solving the inclusion:
        
    \vspace{-0.9cm}
    
    \begin{align}
		0 \in A(\bar{x},\bar{y}) + B(\bar{x},\bar{y}) +  C(\bar{x},\bar{y}), \label{cond}
	\end{align}

    \vspace{-0.9cm}
    
	\noindent where $\mathbb{H} = \mathbb{R}^n \times \mathbb{R}^{m}$ and the  three  operators are
    
     \vspace{-1.1cm}
	
    \begin{align}
		&A: \mathbb{H} \to \mathbb{H}:\;  (x,y) \mapsto (\nabla_{x}F(x,y), - \nabla_{y}F(x,y))    \label{condbis}\\ 
		&B: \mathbb{H} \to \mathbb{H}: \; (x,y) \mapsto  (\nabla \varphi_1(x), \nabla \psi_1(y)), \nonumber\\
		&C: \mathbb{H} \to 2^{\mathbb{H}}: \;  (x,y) \mapsto  \partial \varphi_2(x) \times \partial \psi_2(y).\nonumber
	\end{align}

 \vspace{-0.5cm}

\noindent Next, we give two examples  of this minimax problem.

\noindent a) For simplicity, let us consider $m=1$ and
\begin{equation}\label{e:Fnm}
(\forall (x,y)\in \mathbb{H})\quad
F(x,y) =  y g(x), 
\end{equation}

 \vspace{-0.3cm}
 
\noindent where  $g$ is a twice differentiable convex function, which has a $\nu$-H\"older continuous gradient with constant $L_{g}$,  $\varphi_{1}$ and   $\psi_{1}$ are convex functions, $\psi_{2}$ is the indicator function of  the interval $[0,+\infty[$ and $\varphi_{2} = t_{1} + t_2$, where $t_{1}$ is the indicator function of a nonempty closed convex set $D$, and $t_2$ is a proper lower-semicontinous  convex function which is Lipschitz continuous on its domain  with modulus $L_{t_2}>0$.  We assume that $D \subseteq \operatorname{dom} t_{2}$ and there is a point in the intersection of the relative interiors of $D$ and $\operatorname{dom}t_{2}$. Let us show that Assumption \ref{ass1} holds for this example.

\noindent (i) - (iii) From the definitions of operators $A$, $B$ and $C$, we can easily verify that $C$ is maximally monotone, $A$ is a continuous single-valued operator, and $B$ is Lipschitz continuous.

\noindent (iv) Operator $A$ is given by
    \begin{equation}\label{e:AFnm1}
			A: \mathbb{H} \to \mathbb{H}: \;\; (x,y) \mapsto \left(  y \nabla g(x), -g(x)  \right). 
	\end{equation}
    
     \vspace{-0.3cm}
     
\noindent    The Jacobian $\mathcal{J}_A$ of $A$ at $(x,y) \in \mathbb{H}$ is

 \vspace{-0.5cm}
    
    \begin{align*}
      \mathcal{J}_{A}(x,y) =    \begin{bmatrix}
		 y \nabla^{2}g(x) & \nabla g(x)\\
			-\nabla g(x)^\top  & 0 
		\end{bmatrix}.
  \end{align*}

 \vspace{-0.5cm}
  
     \noindent Note that, for every $(x,y) \in \mathbb{R}^n\times [0,+\infty[$, $\mathcal{J}_{A}(x,y)$ is positive semidefinite matrix. Since $A$ is continuous and monotone on $\mathbb{R}^n\times [0,+\infty[$, then $A$ is maximally monotone 
    on $\mathbb{R}^n\times [0,+\infty[$, see \cite{BauCom:11}. Moreover, since $\varphi$ and $\psi_{1}$ are proper lower semicontinuous convex functions, then $A+B+C$ is a monotone operator, which is a particular  instance of a pseudo-monotone operator.      

\noindent (v)   Note that in this case $\text{dom}\, C = D \times [0, \infty[$ and we have: 
\begin{equation*}
    J_{\gamma C}(x,y) = (\operatorname{prox}_{\gamma(t_{1}+t_{2})}(x),\operatorname{proj}_{[0,+\infty[}(y)).
\end{equation*}
Following Example \ref{ex:3terms}, the following relation holds:
\vspace{-0.5cm}
 \begin{align*}
     \text{prox}_{\gamma(t_{1}+t_{2})}(x) = \text{proj}_{D} \circ \text{prox}^{t_{1}}_{\gamma t_{2} }, \; \text{where} \;\;
        \text{prox}^{t_{1}}_{\gamma t_{2} } = (\text{Id} + \gamma   \partial t_{2}   \circ \text{proj}_{D})^{-1}.
             \vspace{-0.1cm}
        \end{align*} 

\vspace{-0.5cm}
\noindent Let $(u,w)\in \mathbb{H}^2$,  $q = \text{proj}_{\text{dom} C}w$, and 
$z= q - \gamma u$. Considering the decomposition  $w=(w_{1},w_{2}), u=(u_{1},u_{2})$  and $z=(z_{1},z_{2})$.  Define $p=\text{prox}^{t_{1}}_{\gamma t_{2}}(z_{1})$. Then, $p + \gamma \hat{p}=z_{1}$, for some $\hat{p} \in  \partial t_{2}   (\text{proj}_{D} p)$, and  $\|\hat{p}\| \leq L_{t_2} $. We further  have
\vspace{-0.5cm}
\begin{align*}
   \|\text{proj}_{\text{dom} C}w  - J_{\gamma C} z\|^2 
   &= \| \text{proj}_{D}w_{1} - \text{proj}_{D} \circ \text{prox}^{t_{1}}_{\gamma t_{2} }z_{1}\|^2 + \|\text{proj}_{[0,+\infty[}w_{2} - \text{proj}_{[0,+\infty[}z_{2} \|^2 \\
   & \leq \| \text{proj}_{D}w_{1} - \text{prox}^{t_{1}}_{\gamma t_{2} }z_{1}\|^2 + \|\text{proj}_{[0,+\infty[}w_{2} - z_{2}\|^2 \\
   &= \|\text{proj}_{D} w_{1} - z_{1} + \gamma \hat{p}\|^2 + \|\text{proj}_{[0,+\infty[}w_{2} - z_{2}\|^2 \\
   &\leq 2\|\text{proj}_{D} w_{1} - z_{1}\|^2 + 2\gamma^2 \|\hat{p}\|^2 + \|\text{proj}_{[0,+\infty[}w_{2} - z_{2}\|^2 \\
   & = 2 \gamma^2\|u_{1}\|^2 + 2\gamma^2 \|\hat{p}\|^2 + \gamma^2 \|u_{2}\|^2  \leq 2 \gamma^2 \|u\|^2 + 2\gamma^2 L_{t_{2}}^2.
\end{align*}

\vspace{-0.5cm}
\noindent Thus, we  get  $\zeta=\sqrt{2}$ and $\tau = \sqrt{2}L_{t_{2}}$.

\noindent  (vi) For every $z = (x,y)\in \mathbb{R}^{n} \times \mathbb{R}$ and $\bar{z} = (\bar x,\bar y)\in \mathbb{R}^{n} \times \mathbb{R}$, we have
   \vspace{-0.5cm}
   \begin{align*}
    \|Az  - A\bar z\|^2 
    &\leq 2\| \nabla g(x)\|^2\|z - \bar z\|^2 + 4L_g^2 \|z - \bar z\|^{2+2\nu} 
    + 4L_{g}^2|y|^2 \|z -\bar z\|^{2\nu}.  \vspace{-0.1cm}
    \end{align*}
%\begin{proof}

 \vspace{-0.5cm}   
 \noindent Indeed, from the definition of $A$ (see \eqref{e:AFnm1}),
  \vspace{-0.5cm}
  \begin{align}
    \|A(x,y) - A(\bar x, \bar y)\|^2 = \|\nabla g(x)y - \nabla g(\bar x) \bar y\|^2 + |g(\bar x) - g(x)|^2. \label{eq:28}
     \vspace{-0.1cm}
  \end{align}

  \vspace{-0.5cm}
  \noindent Moreover,
   \vspace{-0.5cm}
  \begin{align}
      \|\nabla g(x)y - \nabla g(\bar x) \bar y\|^2 &= \|\nabla g(x)y - \nabla g(x) \bar y + \nabla g(x) \bar y - \nabla g(\bar x) \bar y\|^2 \nonumber\\
   %   &\leq 2\|\nabla g(x)y - \nabla g(x) \bar y\|^2 + 2 \|\nabla g(x)\bar y - \nabla g(\bar x) \bar y\|^2\nonumber\\
      &\leq 2 \|\nabla g(x)\|^2 |y-\bar y|^2 + 2|\bar y|^2 \|\nabla g(x) - \nabla g(\bar x)\|^2 \nonumber \\
      &\stackrel{(\ref{eq:holder})}{\leq} 2 \|\nabla g(x)\|^2 |y-\bar y|^2 + 2L_g^2|\bar y|^2 \|x-\bar x\|^{2\nu} \nonumber\\
      &\leq 2 \|\nabla g(x)\|^2 |y-\bar y|^2 + 4L_g^2|y|^2 \|x-\bar x\|^{2\nu}  + 4L_g^2|y - \bar y|^2 \|x-\bar x\|^{2\nu}, \label{eq:29}
       \vspace{-0.1cm}
  \end{align}
  \noindent where, in the first and last inequalities, we used the fact that $\|a + b\|^2 \leq 2\|a\|^2 + 2\|b\|^2$, for all $a,b \in \mathbb{R}^{n}$ and $n \in \mathbb{N}$. On other hand, from \eqref{eq:holderineq} we deduce that
  \vspace{-0.5cm}
    \begin{align}
       |g(\bar x) - g(x)|^2 &\leq 2 |\langle \nabla g(x), \bar x-x\rangle|^2 +\dfrac{2L_g^2}{(1+\nu)^2} \|\bar x-x\|^{2+2\nu} \nonumber \\
       &\leq 2 \| \nabla g(x)\|^2 \|\bar x-x\|^2 + \dfrac{2L_g^2}{(1+\nu)^2} \|\bar x-x\|^{2+2\nu}\nonumber\\
        &\leq 2 \| \nabla g(x)\|^2 \|\bar x-x\|^2 + 4L_g^2 \|\bar x-x\|^{2+2\nu}.
       \label{eq:30}
    \end{align}

\vspace{-0.5cm}
 \noindent Altogether, \eqref{eq:28}, \eqref{eq:29}, and \eqref{eq:30} lead to
 \vspace{-0.5cm}
    \begin{align*}
    &\|A(x,y) - A(\bar x, \bar y)\|^2\\
    &\leq 2\|
    \nabla g(x)\|^2\|(x,y) \!-\! (\bar x, \bar y)\|^2
    \!+\! 4L_g^2 \|(x,y) \!-\! (\bar x,\bar y)\|^2 \|\bar x-x\|^{2\nu} \!+\! 4L_g^2|y|^2 \|x- \bar x\|^{2\nu} \\
    &\leq 2\| \nabla g(x)\|^2\|z - \bar z\|^2 + 4L_g^2\|z - \bar z\|^{2+2\nu}+4L_g^2|y|^2 \|x- \bar x\|^{2\nu}.
    \end{align*}

\vspace{-0.5cm}
\noindent Hence,  Assumption \ref{ass1}.\ref{ass1vi} holds with   $\mu = 2 \nu$, $\beta=2+2 \nu$ and $\theta=2$. 
  \noindent Finally, the inequality $\|x- \bar x\| \leq \|(x,y) - (\bar x,\bar y)\|$ allows us to prove the statement.

\medskip
\noindent b) Consider \begin{equation}
(\forall (x,y)\in \mathbb{H})\quad
F(x,y) = \sum_{i=1}^{m} y_{i} g_{i}(x), 
 \vspace{-0.2cm}
\end{equation}
where  $(g_{i})_{1\le i\le \bar{m}}$ are twice differentiable convex functions, which have Lipschitz continuous gradients with constants $L_{g_i}>0$, $(g_{i})_{\bar{m}+1\le i\le m}$ are the following affine functions:
\[
(\forall i \in \{\bar{m}+1,\ldots,m\})\quad
g_{i}\colon x \mapsto l_{i}^{\top}x - r_{i}
\]
 with $(l_{i},r_i)_{\bar{m}+1\le i\le m}$ vectors in $\mathbb{R}^n\times \mathbb{R}$. In addition,  $\varphi_{1}$ and   $\psi_{1}$ are convex functions, $\psi_{2}$ is the indicator function of  the set $[0,+\infty[^{\bar{m}} \times \mathbb{R}^{m-\bar{m}}$, and $\varphi_{2} = t_{1} + t_2$, where $t_{1}$ is the indicator function of a nonempty closed convex set $D$ and $t_2$ is a proper lower-semicontinous  convex function which is Lipschitz on its domain  with modulus $L_{t_2}>0$.  We assume that $D \subseteq \operatorname{dom} t_{2}$ and there is a point in the intersection of the relative interiors of $D$ and $\operatorname{dom}t_{2}$. Let us show that Assumption \ref{ass1} holds for this example.

\noindent (i)-(iii) From the definitions of the operators $A$, $B$, and $C$, it is straightforward that $C$ is maximally monotone, $A$ is a continuous single-valued operator, and $B$ is Lipschitz continuous.

\noindent (iv) First note that $A + \{ \mathbf{0_n} \} \times \mathcal{N}_{[0,+\infty[^m]}$ is maximally monotone. Indeed, if we consider the notation $A = A_1 \times A_2$, where, for every $(x,y)\in \mathbb{R}^n\times \mathbb{R}^m$, $A_{1}(x,y) \in \mathbb{R}^{n}$ and $A_{2}(x,y) \in \mathbb{R}^{m}$, then $(A + \{\mathbf{0_n}\} \times \mathcal{N}_{[0,+\infty[^m}) (x,y) = A_1(x,y) \times [A_2(x,y) +   \mathcal{N}_{[0,+\infty[^m}(y)]$. Moreover, we have
    \begin{equation}\label{e:AFnm}
			A: \mathbb{H} \to \mathbb{H}: \;\; (x,y) \mapsto \Big( \underbrace{\sum_{i=1}^m y_i \nabla g_i(x)}_{A_1(x,y)}, \underbrace{-g(x)}_{A_2(x,y)}  \Big), 
	\end{equation}
where $g(x) = [g_1(x),\ldots,g_m(x)]^\top$.
The Jacobian $\mathcal{J}_A$ of $A$ at $(x,y)$ is
    \begin{align*}
      \mathcal{J}_{A}(x,y) =    \begin{bmatrix}
			\sum_{i=1}^m y_i \nabla^{2}g_i(x) & \nabla g(x)\\
			-\nabla g(x)^\top  & \mathbf{0}_{m\times m} 
		\end{bmatrix},
  \;\;\text{with}\;\;
  \nabla g(x) =[\nabla g_1(x),\ldots,\nabla g_m(x)].
  \end{align*}
     \noindent  Note that, for every $(x,y) \in \mathbb{R}^n \times ([0,+\infty[^{\bar{m}} \times \mathbb{R}^{m-\bar{m}})$, $\mathcal{J}_{A}(x,y)$ is a positive semidefinite matrix. Since $A$ is continuous and monotone on $\mathbb{R}^n \times ([0,+\infty[^{\bar{m}} \times \mathbb{R}^{m-\bar{m}})$, then $A$ is maximally monotone 
    on $\mathbb{R}^n\times ([0,+\infty[^{\bar{m}} \times \mathbb{R}^{m-\bar{m}})$, see \cite{BauCom:11}. Moreover, since $\varphi$ and $\psi_{1}$ are proper lower semicontinuous convex functions, then $A+B+C$ is a monotone operator. 

\noindent(v) Following the same argument as in the case $m=1$, we  get $\zeta=\sqrt{2}$ and $\tau = \sqrt{2}L_{t_{2}}$.

\noindent (vi) For every $z = (x,y)\in \mathbb{R}^{n} \times \mathbb{R}^{m}$ and $\bar{z} = (\bar{x},\bar{y}) \in \mathbb{R}^{n} \times \mathbb{R}^{m}$, we have

\vspace{-0.5cm}
\begin{align}
    \|A(x,y) - A(\bar{x},\bar{y})\|^2 \leq {a}(x,y) \|(x,y) - (\bar{x},\bar{y})\|^2 + {b}\|(x,y) - (\bar{x},\bar{y})\|^4, \label{eq:34}
\end{align}

\vspace{-0.5cm}
\noindent with $\displaystyle {b} = \dfrac{5}{2}\sum_{i=1}^{m} L_{g_i}^2 $, ${a}(x,y) = 2(\rho(x,y) +  \sum_{i=1}^{m} \| \nabla g_{i}(x)\|^2)$, and
\begin{align*}
%    b(x,y) &= \rho(x,y) +  2 \sum_{i=1}^{m} \| \nabla g_{i}(x)\|^2\\
    \rho(x,y) &= 2 \max\left(m \max_{1\le i\le m} \|\nabla g_{i}(x)\|^2, \left(\sum_{i=1}^{m} L_{g_i}|y_{i} |\right)^2\right).
\end{align*}
%\begin{proof}
 \noindent Indeed, similarly to the previous example,
  \begin{align}
    \|A(x,y) - A(\bar{x},\bar{y})\|^2 = \left\|\sum_{i=1}^{m}\nabla g_{i}(x)y_{i} - \nabla g_{i}(\bar{x}) \bar{y}_{i}\right\|^2 + \sum_{i=1}^{m} |g_{i}(\bar{x})- g_{i}(x)|^2. \label{eq:31}
  \end{align}
\noindent Moreover,
\vspace{-0.5cm}
  \begin{align}
      &\left\|\displaystyle \sum_{i=1}^{m}\left(\nabla g_{i}(x)y_{i} - \nabla g_{i}(\bar{x}) \bar{y}_{i}\right)\right\|^2 = \left\|\sum_{i=1}^{m}\left(\nabla g_{i}(x)y_{i} - \nabla g_{i}(x) \bar{y}_{i} + \nabla g_{i}(x)\bar{y}_{i} - \nabla g_{i}(\bar{x}) \bar{y}_{i}\right)\right\|^2 \nonumber\\
   %   & \leq \left(\sum_{i=1}^{m} \|\nabla g_{i}(x)y_{i} - \nabla g_{i}(x) \bar{y}_{i}\| +  \|\nabla g_{i}(x)\bar{y}_{i} - \nabla g_{i}(\bar{x}) \bar{y}_{i}\|      \right)^2 \nonumber\\
      & \leq \left(\sum_{i=1}^{m} \|\nabla g_{i}(x)\| |y_{i} - \bar{y}_{i}| + \sum_{i=1}^{m}|\bar{y}_{i} | \|\nabla g_{i}(x)- \nabla g_{i}(\bar{x})\|      \right)^2 \nonumber \\
      & \leq \left(\max_{1\le i\le m} \|\nabla g_{i}(x)\| \sum_{i=1}^{m} |y_{i} - \bar{y}_{i}| + \sum_{i=1}^{m} L_{g_i}|\bar{y}_{i} | \|x-\bar{x}\|      \right)^2 \nonumber \\
     % & \leq \! \left(\! \sqrt{m} \! \max_{1\le i\le m} \! \|\nabla g_{i}(x)\| \|y - \bar{y}\| +\! \sum_{i=1}^{m} \! L_{g_i}|y_{i} | \|x-\bar{x}\| + \! \sum_{i=1}^{m} \!L_{g_i}|\bar{y}_{i} - y_{i} | \|x-\bar{x}\| \! \right)^2\nonumber \\
       &\leq \! 2 \! \left( \! \sqrt{m} \! \max_{1\le i\le m} \! \|\nabla g_{i}(x)\| \|y \!-\! \bar{y}\| \!+\! \sum_{i=1}^{m} \!  L_{g_i}|y_{i} | \|x \!-\! \bar{x}\| \! \right)^2 \!\!\!+ \! 2 \! \left(\sum_{i=1}^{m} \! L_{g_i}|\bar{y}_{i} \!-\! y_{i} | \|x \!-\! \bar{x}\| \! \right)^2, \nonumber
  \end{align}

  \vspace{-0.5cm}
    \noindent where, in the last inequality, we  used the Cauchy-Schwarz inequality.
    %the fact that $\sum_{i=1}^{m}|x_{i}| \leq \sqrt{m} \|x\|$ and in the last one, we have used $(a + b)^2 \leq 2a^2 + 2 b^2$. 
    %It  then follows that:
    Hence,
      \begin{align}
       \left\|\sum_{i=1}^{m}\nabla g_{i}(x)y_{i} - \nabla g_{i}(\bar{x}) \bar{y}_{i}\right\|^2 & \leq \rho(x,y)\left(\|y - \bar{y}\| + \|x-\bar{x}\|\right)^2 
       + 2\left(\sum_{i=1}^{m} L_{g_i}^2\right) \|\bar{y} - y\|^2  \|x-\bar{x}\|^2\nonumber\\
       %&+4 \max\left(m \max_{i=1:m} \|\nabla g_{i}(x)\|^2, \left(\sum_{i=1}^{m} L_{g_i}|y_{i} |\right)^2\right)\left(\|y - \bar{y}\|^2 + \|x-\bar{x}\|^2\right) \nonumber \\
      & \leq2\rho(x,y) \|(x,y) - (\bar{x},\bar{y})\|^2 + 2\left(\sum_{i=1}^{m} L_{g_i}^2 \right)\|(x,y) - (\bar{x},\bar{y})\|^4.\label{eq:32}
  \end{align}

  \vspace{-0.5cm}
   \noindent On other hand, using \eqref{eq:holderineq} with $\nu=1$, 
    \vspace{-0.2cm}
     \begin{align}
       \sum_{i=1}^{m} |g_{i}(\bar{x})- g_{i}(x)|^2&\leq \sum_{i=1}^{m} 2 |\langle \nabla g_{i}(x), \bar{x}-x\rangle|^2 + {\sum_{i=1}^{m}} \dfrac{L_{g_i}^2}{2}  \|x-\bar{x}\|^{4} \nonumber \\
       &\leq 2 \sum_{i=1}^{m} \| \nabla g_{i}(x)\|^2 \|x-\bar{x}\|^2 + \sum_{i=1}^{m} \dfrac{L_{g_i}^2}{2} \|x-\bar{x}\|^{4}. \label{eq:33}
       \vspace{-0.2cm}
    \end{align}

    \vspace{-0.5cm}
   \noindent Hence, \eqref{eq:31}, \eqref{eq:32}, \eqref{eq:33}, and the fact $\|x-\bar{x}\|\leq \|(x,y) - (\bar{x},\bar{y})\|$ yield \eqref{eq:34}. From \eqref{eq:34} it follows that  Assumption \ref{ass1}.\ref{ass1vi} holds with   $\mu = 2$ and $\theta=4$ (note that in this case $c(z_1) = 0$). 
\end{example}

\noindent \begin{remark}
    
 One concrete application of the {example} above is the quadratically constrained quadratic program (QCQP) problem: 
 \vspace{-0.5cm}
        \begin{align}
            & \min_{x \in \mathbb{R}^{n}, x \geq 0} \, \frac{1}{2}x^\top Q_0 x + b^\top x + c \nonumber\\
            & \text{s.t. } \; \frac{1}{2}x^\top Q_{i}x + l^\top_{i}x \leq  r_{i} \; \forall i\in\{1,\ldots,\bar{m}\}, \quad l^\top_{i}x =  r_{i} \; \forall i\in\{\bar{m}+1,\ldots,m\},   \label{eq:qcqp}
        \end{align}
        where  $(Q_{i})_{0\le i\le \bar{m}}$ are  positive semidefinite  matrices of dimension $n \times n$, and $(l_{i})_{1\le i\le m}$ and $b$ are vectors in $\mathbb{R}^n$. Rewriting the QCQP into the Lagrange primal-dual form 
        %i.e., $F(x,y)$ is the  Lagrangian (which is not bilinear), 
        using the dual variables $y=(y_i)_{1\le i\le m}$, we get
        \vspace{-0.5cm}
        \begin{align}
        \label{opA:QCQP}
			&A: \mathbb{H} \to \mathbb{H}: \; (x,y) \mapsto \left(\sum_{i=1}^{m} (Q_{i}x + l_{i}) y_{i},\left(- \frac{1}{2}x^\top Q_{i}x - l^\top_{i}x +  r_{i}\right)_{1\leq i \leq m}\right) \nonumber\\
            &B: \mathbb{H}  \to \mathbb{H}:  \; (x,y) \mapsto (Q_0x + b,\mathbf{0}_m)\\
            &C: \mathbb{H}  \to 2^{\mathbb{H}}:   \;  (x,y) \mapsto (\mathcal{N}_{[0,+\infty)^{n}}(x) 
 \times (\mathcal{N}_{[0,+\infty)^{\bar{m}}}((y_i)_{1\le i \le \bar{m}}) \times \{\mathbf{0}_{m-\bar{m}}\}), \nonumber
	\end{align}

\vspace{-0.5cm}
\noindent where  we have set $\mathbb{H}= \mathbb{R}^{n} \times \mathbb{R}^{m}$ and $Q_{i} = 0$ for every $i\in \{\bar{m}+1,\ldots,m\}$. QCQP's have many applications, e.g.,  in signal processing \cite{HuaPal:14},  triangulation in computer vision \cite{AhoAgaTho:12}, semi-supervised learning \cite{CheLi:21}, learning of kernel matrices \cite{Lan:02},  or steering direction estimation for RADAR detection~\cite{Mai:11}.
\end{remark}

   \begin{example}\label{ex:frac} 
    \noindent Consider the following problem:
    \vspace{-0.5cm}
    \begin{align}
        &\min_{ x \in D} f(x), \label{constProb} 
    \end{align}

    \vspace{-0.5cm}
    \noindent where $D$ is a nonempty closed convex subset of $\mathbb{H} = \mathbb{R}^n$ and $f:\mathbb{R}^{n} \to \mathbb{R}$ is a continuously differentiable pseudo-convex function on $D$. Some examples of pseudo-convex functions are encountered in fractional programs, see below and also  \cite{Sch:81,HasJad:07}.
    If $\bar{x}$ satisfies  the first-order optimality condition for 
    \eqref{constProb}, then we have the following inclusion:
     \vspace{-0.5cm}
    \begin{align}
    \label{prob1}
		0 \in A\bar{x} + C\bar{x},
     \vspace{-0.1cm}
	\end{align}

    \vspace{-0.5cm}
    \noindent where the operators are  
    \begin{equation}
    \label{eq:AQuadFrac}
    Ax= \begin{cases}
    \nabla f(x) &  \text{if} \; x \in D  \\
    \varnothing & \text{otherwise,} 
    \end{cases}
   \quad  {B=0,}\;\;  \text{ and } \;\;  C=  \mathcal{N}_{D}.
    \end{equation}

    \noindent  (Fractional programming) Consider the following quadratic  fractional programming problem: 
    \begin{align}
        &\min_{x\in D} f(x)  \;   :=  \dfrac{ \frac{1}{2}x^\top Q x - h^{\top}x +  h_{0}}{d^\top x + d_0}   \quad \text{with} \quad D= \{x \in \mathbb{R}^{n}\mid d^\top x\geq 0\}, \label{QuadFracPro}
    \end{align}
    \noindent {where $f:D \to \mathbb{R}$}, $d_{0}\in ]0,+\infty[$, $h_{0} \in \mathbb{R}$, $(h,d) \in (\mathbb{R}^{n})^2$, and $Q \in \mathbb{R}^{n\times n}$ is a symmetric matrix.
     When the matrix $Q$ is positive semidefinite, the function $f$ is pseudo-convex since it is the ratio of convex over concave functions, see \cite{BorLew:06}. Otherwise, \cite{CamCroMar:02,CarMar:07} present necessary and sufficient conditions for the function $f$ to be pseudo-convex. Particular cases of \eqref{QuadFracPro} are problems whose objective is a sum of a linear and a linear fractional function, i.e., when $Q = (rd^{\top} + dr^{\top})/2$, which yields the following formulation:
 \vspace{-0.5cm}
   \begin{align}
        &\min_{x\in D} f(x)  \;   := r^\top x + \dfrac{{h}^\top x + h_{0}}{d^\top x + d_0} \quad \text{with} \quad D= \{x \in \mathbb{R}^{n}\mid d^\top x\geq 0\}. \label{LinFracPro}
        \end{align}
    
    \vspace{-0.5cm}
    \noindent Indeed,  setting $Q = (rd^\top + dr^\top)/2$ in \eqref{QuadFracPro} yields
    \begin{align*}
      \dfrac{ \frac{1}{2}x^\top Q x - h^{\top}x +  h_{0}}{d^\top x + d_0} &= \dfrac{ \frac12 x^{\top}rd^\top x  - h^{\top}x +  h_{0} }{d^\top x + d_0} =
      \frac12 r^{\top}x + \dfrac{h_{0} -(h + \frac{d_{0}}{2}r)^{\top}x}{d^\top x + d_0}.
    \end{align*}
     Then, by defining $\bar{r} = r/2$ and $\bar{h} = - h - \frac{d_{0}}{2} r$, we obtain a problem of the form  \eqref{LinFracPro}. Reference \cite{MarCar:12} presents several cases when  $f$ is pseudo-convex over the polyhedral set $\{x\in\mathbb{R}^{n}\mid d^\top x + d_{0} > 0$\}, namely,  if $r = \eta d$, with $\eta \geq 0$, or  $h=\zeta d$, with $h_0 - \zeta d_{0} \geq 0$ (see \cite[Theorem 1]{MarCar:12} for more details). Fractional programming arises  e.g., in portfolio   and  transportation problems (see \cite{MarCar:12} for more details). 
    \end{example} 
Below, we show that operators $A$, $B$, and $C$ defined in \eqref{eq:AQuadFrac} satisfy Assumption \ref{ass1}.

\noindent (i) - (iii) From the definitions of the operators $A$, $C$, we can easily see that $C$ is maximally monotone and $A$ is a continuous single-valued operator.
    
\noindent (iv) Let us show that $A+C$ is pseudo-monotone. Indeed,   consider $(x,w) \in D^2$. Assume that 
    \vspace{-0.5cm}
    \begin{align}
    (\forall \hat{x} \in Cx)\quad 
        \langle \nabla f(x) + \hat{x}, w-x \rangle \geq 0.
        \label{eq:24}
    \end{align}

    \vspace{-0.5cm}
    \noindent We need  to show that
    \vspace{-0.5cm}
     \begin{align}
     (\forall \hat{w} \in Cw)\quad
        \langle \nabla f(w) + \hat{w}, w-x \rangle \geq 0.  \label{eq:25}
    \end{align}

    \vspace{-0.5cm}
    \noindent It follows from the definition of the normal cone in \eqref{eq:normalcone} that
    \vspace{-0.5cm}
    \begin{align}
    (\forall \hat{x} \in Cx)\quad
        \langle \hat{x}, w-x \rangle \leq 0 \quad \text{and} \quad
    (\forall \hat{w} \in Cw)\quad 
        \langle \hat{w}, x-w \rangle \leq 0. \label{eq:27}
    \end{align}

    \vspace{-0.5cm}
    \noindent Combining \eqref{eq:24} and the first inequality in \eqref{eq:27} yields
     \vspace{-0.5cm}
    \begin{align*}
        \langle \nabla f(x), w-x \rangle \geq 0 .
    \end{align*}

     \vspace{-0.5cm}
    \noindent Since $f$ is pseudo-convex, then the above inequality implies that 
    \vspace{-0.5cm}
    \begin{align*}
        \langle \nabla f(w), w-x \rangle \geq 0. 
    \end{align*}

    \vspace{-0.5cm}
    \noindent Hence, from the previous inequality and the second one in \eqref{eq:27}, we derive \eqref{eq:25}. Therefore, Assumption \ref{ass1}.\ref{ass1iv} holds. 

\noindent (v) Note that $J_{\gamma C} z = \text{proj}_{\text{dom} C}(z)$. Using the nonexpensiveness of the projection operator, since $q = \text{proj}_{\text{dom} C}w$ and
		$z= q - \gamma u$, we get
        \vspace{-0.5cm}
		\begin{align*}
			\|\text{proj}_{\text{dom} C}(w) - \text{proj}_{\text{dom} C}(z)\| &\leq \|q-z\| =  \gamma \|u\|.
		\end{align*}

        \vspace{-0.5cm}
		\noindent Hence, in this case the inequality holds with $\zeta = 1$ and any $ \tau>0$.

\noindent (vi) Before showing that this assumption is valid for \eqref{QuadFracPro}, let us fist derive a more focused inequality for  \eqref{LinFracPro}.

\begin{enumerate}
\medskip 
%    \begin{example}
 \item Consider problem \eqref{LinFracPro}  and operator $A$ defined in \eqref{eq:AQuadFrac}. 
 The Hessian of $f$ is
    \vspace{-0.5cm}
        \begin{align*}
        (\forall x \in D)\quad 
            \nabla^2 f(x) = \dfrac{2 ( h^{\top}x + h_{0})}{(d^{\top}x + d_{0})^3} dd^\top - \dfrac{1}{(d^{\top}x + d_{0})^2}(d  h^{\top} +  hd^{\top}). 
        \vspace{-0.2cm}
        \end{align*}

        \vspace{-0.5cm}
       \noindent Consider $(x, \bar x) \in D^2$. By the mean value inequality, there exists $w \in (x,\bar x)$ s.t. 
       \vspace{-0.1cm}
        \begin{align*}
           &\|\nabla f (x) - \nabla f (\bar x) \|^2 \leq \|\nabla^2 f(w)\|^2 \|x - \bar x\|^2 \\
            &\leq \left( \dfrac{2}{(d^{\top}w + d_{0})^3}   \|d\|^2 | h^{\top}w + h_{0}| + \dfrac{2}{(d^\top w + d_{0})^2} |d^{\top} h|\right)^2 \|x - \bar x\|^2 \\ 
            & \leq \left( \dfrac{2}{d_{0}^3} \|d\|^2| h^{\top}(x-w)|  +\dfrac{2}{d_{0}^3} \|d\|^2 | h^{\top}x + h_{0}| + \dfrac{2}{d_{0}^2} |d^{\top} h| \right)^2 \|x - \bar x\|^2 \\
            & \leq 2\left[ \left( \dfrac{2}{d_{0}^3} \|d\|^2| h^{\top}(x-w)| \right)^2 + \left(\dfrac{2}{d_{0}^3} \|d\|^2 | h^{\top}x + h_{0}| + \dfrac{2}{d_{0}^2} |d^{\top} h| \right)^2 \right] \|x - \bar x\|^2 \\
            &\leq \dfrac{8}{d_0^6} \|d\|^4\|h\|^2\|x - \bar x\|^4 + \dfrac{8}{d_0^4}\left(\dfrac{\|d\|^2}{d_{0}}  |h^{\top}x + h_{0}| + |d^{\top} h| \right)^2 \|x - \bar x\|^2, \vspace{-0.2cm}
        \end{align*}

      \vspace{-0.5cm}
      \noindent where in the third inequality we used the fact that, since $D$ is convex, $w \in D$, hence $d^{\top}w \geq 0$, in the fourth inequality, we used the convexity of $(\cdot)^2$, and in the last one we used  that $w \in (x,\bar x)$. Hence,  Assumption \ref{ass1}.\ref{ass1vi} holds with  ${a}(z_{1}) = \dfrac{8}{d_0^4}\left(\dfrac{\|d\|^2}{d_{0}}  | h^{\top}z_{1} + h_{0}| + |d^{\top}h| \right)^2$,  ${b}(z_{1}) = \dfrac{8}{d_0^6} \|d\|^4\| h\|^2$,  ${c}(z_{1}) = 0$, $\mu = 2$, and $\theta=4$.
  %  \end{example}
   \item {
\noindent Consider problem \eqref{QuadFracPro}  and operator $A$ defined in \eqref{eq:AQuadFrac}. The Hessian of $f$ is 
        \vspace{-0.5cm}
        \begin{align*}
        (\forall x \in D)\quad 
            \nabla^2 f(x) = \dfrac{Q}{d^{\top}x + d_{0}} + \dfrac{ (2f(x)dd^{\top} - (Qx - h)d^{\top} - d(Qx - h)^{\top})}{(d^{\top}x + d_{0})^2}.
        \vspace{-0.4cm}
        \end{align*}

        \vspace{-0.5cm}
         \noindent Consider $(x, \bar x) \in D^2$. By the mean value inequality, there exists $w \in (x,\bar x)$ s.t.
        \vspace{-0.5cm}
        \begin{align}
           &\|\nabla f (x) - \nabla f (\bar x) \| \leq \|\nabla^2 f(w)\| \|x - \bar x\|.  \label{eq:35}
        \vspace{-0.2cm}
        \end{align} 

        \vspace{-0.5cm}
        \noindent After some calculations similar to those in the previous example, we get
        \vspace{-0.5cm}
        \begin{align}
            &\|\nabla^2 f(w)\| \leq  \dfrac{1}{d^{\top}w + d_{0}}\|Q\| +  \dfrac{|w^\top Q w - 2h^{\top}w +  2h_{0}|}{(d^{\top}w + d_{0})^3 } \|d\|^2  + \dfrac{ 2\|d\|\|Qw-h\| }{{(d^{\top}w + d_{0})^2}} \nonumber   \\
            &\leq \dfrac{\|Q\|}{d_{0}} +  \dfrac{\|d\|^2}{d_{0}^3} ( \|Q\|\|w\|^2 + 2| h^{\top}w -  h_{0}|) + \dfrac{ 2\|d\|\|Qw-h\| }{{ d_{0}^2}} \nonumber \\
            &\leq \dfrac{\|Q\|}{d_{0}} +  \dfrac{\|d\|^2}{d_{0}^3} ( 2\|Q\|\|w-x\|^2 + 2\|Q\|\|x\|^2  + 2|h^{\top}(w-x)| + 2|h^{\top}x -  h_{0}|) \nonumber \\
            &\quad + \dfrac{ 2\|d\| }{{ d_{0}^2}}(\|Qx-h\| + \|Q\| \|w-x\|)  \nonumber \\
            &\leq \dfrac{\|Q\|}{d_{0}} +  \dfrac{2\|d\|^2}{d_{0}^3} ( \|Q\|\|\bar{x}-x\|^2 + \|Q\|\|x\|^2  + \|h\|\|\bar{x}-x\| + |h^{\top}x -  h_{0}|) \nonumber\\
            &\quad + \dfrac{ 2\|d\| }{{ d_{0}^2}}(\|Qx-h\| + \|Q\| \|\bar{x}-x\|) \nonumber \\
            & \leq \dfrac{\|Q\|}{d_{0}} +  \dfrac{2\|d\|^2}{d_{0}^3}(\|Q\|\|x\|^2 + |h^{\top}x -  h_{0}|) + \dfrac{ 2\|d\| }{{ d_{0}^2}}\|Qx-h\|
             \nonumber \\
             & \quad + 
              \left( \dfrac{2\|d\|^2}{d_{0}^3} \|h\|
            + \dfrac{ 2\|d\| }{{ d_{0}^2}}\|Q\| \right)\|\bar{x}-x\| + \dfrac{2\|d\|^2}{d_{0}^3}
             \|Q\|\|\bar{x}-x\|^2. \label{eq:36}
             \vspace{-0.2cm}
        \end{align} 
      %  \noindent  \noindent where we have proceeded similarly to the previous example.
      
      \vspace{-0.5cm}
      \noindent We deduce from \eqref{eq:35} and \eqref{eq:36} that
      \vspace{-0.5cm}
        \begin{align*}
           &\|\nabla f (x) - \nabla f (\bar x) \|^2 \\ 
           &\leq  \dfrac{12\|d\|^4}{d_{0}^6}
             \|Q\|^2\|\bar{x}-x\|^6+3 \left( \dfrac{2\|d\|^2}{d_{0}^3} \|h\|
            + \dfrac{ 2\|d\| }{{ d_{0}^2}}\|Q\| \right)^2\|\bar{x}-x\|^4
           \\&+ 3 \left( \dfrac{\|Q\|}{d_{0}} +  \dfrac{2\|d\|^2}{d_{0}^3}(\|Q\|\|x\|^2 + |h^{\top}x -  h_{0}|) + \dfrac{ 2\|d\| }{{ d_{0}^2}}\|Qx-h\| \right)^2 \|\bar{x}-x\|^2.
           \vspace{-0.2cm}
        \end{align*} 
        
        \vspace{-0.5cm}
        \noindent Therefore,  Assumption \ref{ass1}.\ref{ass1vi} is satisfied with  $\mu = 2$, $\theta = 4$, $\beta = 6$,
        and
        \begin{align}
        & c(z_{1}) =  \dfrac{12\|d\|^4}{d_{0}^6}\|Q\|^2,\quad b(z_{1}) = 12 \left( \dfrac{\|d\|^2}{d_{0}^3} \|h\| + \dfrac{\|d\| }{{ d_{0}^2}}\|Q\| \right)^2,\nonumber\\
        &a(z_{1}) = 3 \left( \dfrac{\|Q\|}{d_{0}} +  \dfrac{2\|d\|^2}{d_{0}^3}(\|Q\|\|z_{1}\|^2 + |h^{\top}z_{1} -  h_{0}|) + \dfrac{ 2\|d\| }{{ d_{0}^2}}\|Qz_{1}-h\| \right)^2. 
  \label{eq:42}
        \end{align}}
%        \medskip

\end{enumerate}  

\vspace{-0.5cm}  
\noindent From the previous discussion, one can see that our assumptions cover a broad range of optimization problems arising in applications.  

%%%%%%%%%%%%%%%%%%%%%%%%%%%%%%

\section{An adaptive forward-backward-forward algorithm}
Adaptive methods are widely used in optimization due to their ability to simplify stepsize tuning \cite{BotCseVuo:20, Ton:23, Pet:22}. Unlike previous approaches to the forward-backward-forward algorithm, such as those in \cite{Tse:00}, we introduce two novel adaptive strategies that bypass the need for computationally expensive Armijo-Goldstein stepsize rules, typically employed when the operator is assumed to be continuous \cite{Tse:00, ThoVuo:19}. In both strategies, the stepsize is determined using the current iterate and the parameters that define the operator’s properties.

\subsection{Investigated algorithm}
In this section, we introduce a new algorithm for solving problem  \eqref{prob}. Our algorithm is similar to the forward-backward-forward splitting algorithm in  \cite{Tse:00} as  it also involves two explicit (forward) steps using $A$ and $B$,  and one implicit (backward) step using $C$. However, the novely of our iterative process lies in the adaptive way we choose the stepsize $\gamma_k$ at each iteration $k$, which is adapted to the assumptions considered on the operators $A, B$, and $C$ (see Assumption \ref{ass1}). 

\vspace{0.2cm}

\begin{center}
		\noindent\fbox{%
			\parbox{10 cm}{%
      \small
				\textbf{Adaptive Forward-Backward-Forward Algorithm (AFBF)}:\\
 %   \begin{enumerate}
%				\item  
   1. Choose 
    %$\alpha_{\min},\alpha_{\max} \in (0,1)$ 
    the initial estimate 
				$ x_{0}\in \operatorname{dom} C$. \\
    %Set $p_{-1} = x_{0}$ and $z_{-1} = %x_{0}$.\\
		%		\item 
     2. For $k\geq 0$ do:  
    %update: 
    \begin{itemize}
				\item [(a)] \label{e:sAa}  Compute the stepsize
                $\gamma_{k} > 0$.
				\item [(b)]\label{e:sAb}  $z_{k}  = x_{k}  - \gamma_{k} (A x_{k} + Bx_{k})$
				\item [(c)] \label{e:sAc}  $p_{k} = J_{\gamma_{k}C} z_{k} $
                \item [(d)] \label{e:sAd}  $q_k = p_{k} - \gamma_{k}\left( A p_{k}+ Bp_{k}\right)$
				\item [(e)] \label{e:sAe} $\hat{x}_{k} = q_{k} - z_k+x_k$
                %\gamma_{k}\left( A p_{k}+ Bp_{k} - Ax_{k}-Bx_{k}\right)  $
				\item [(f)] \label{e:sAf}  $x_{k+1} = \text{proj}_{\text{dom}C}(\hat{x}_{k})$
    \end{itemize}
%  \end{enumerate}
		}}
	\end{center}
\medskip
\noindent Typically,  to prove the convergence of a forward-backward-forward splitting  algorithm, one needs the  operators $A$ and $B$ to satisfy a Lipschitz type inequality \cite{Tse:00}:
\vspace{-0.5cm}
\begin{align}
\label{eq:12}
     \gamma_{k}^2 \|Ax_{k}+Bx_{k} - Ap_{k}  - Bp_{k} \|^2 \leq \alpha_{k}  \|x_{k} - p_{k} \|^2, \vspace{-0.2cm}
\end{align}

\vspace{-0.5cm}
\noindent where $\alpha_k \in ]\alpha_{\min},\alpha_{\max}[\subset ]0,1[$, $k\in \mathbb{N}$.  In our case it appears  difficult to find a positive stepsize $\gamma_{k}$ satisfying  \eqref{eq:12} as  the operator $A$ is not assumed to be  Lipschitz.  However, imposing appropriate assumptions on the operator $A$ (e.g., some generalized Lipschitz type inequality as in  Assumption \ref{ass1}.\ref{ass1vi}), we will show that we can ensure \eqref{eq:12}. In the next sections we provide two adaptive choices for $\gamma_k$ that enable us to prove  the convergence of AFBF. %\red{Note that adaptive stepsize strategies have been also considered e.g., in \cite{ChoNec:23} in order to derive a proper Lyapunov function for a coordinate descent type algorithm for solving unconstrained optimization problems, while in the present paper our stepsize choices will  ensure a Lipschitz type inequality of the form \eqref{eq:12} for AFBF algorithm. Hence, since we consider more general problems and algorithms, the convergence proofs below  are different from those in \cite{ChoNec:23}.}  

%%%%%%%%%%%%%%%%%%%

\subsection{First adaptive choice for the stepsize}
In this section, we design a novel strategy to choose $\gamma_k$ when the operator $A$ satisfies Assumption \ref{ass1}.\ref{ass1vi} with $\mu = 2$. Recall that $(\zeta,\tau) \in ]0,+\infty[^2$ are the parameters which satisfying Assumption\ref{ass1}.\ref{ass1v}. 

\vspace{0.3cm}
    
	\begin{center}
     \small
		\noindent\fbox{%
			\parbox{12 cm}{%
				\textbf{Stepsize Choice 1}: \\
   %             \begin{enumerate}
   %             \item 
                1. Choose $0 < \alpha_{\min} \leq \alpha_{\max} <1$ and $\sigma>0$. \\
                2. For $k \geq 0$ do:
                \begin{itemize}
				\item [(a)] \label{e:s1a} Compute $d(x_{k}) =  \zeta \|Ax_k + Bx_k\| + \tau $ and choose $\alpha_{k} \in [\alpha_{\min},\alpha_{\max}]$. 
               \item [(b)]  \label{e:s1b} Choose $\gamma_{k}$ such that
                \begin{equation} \label{eq:gamma}
              \gamma_{k} \in   \begin{cases}
                 [\sigma,\bar{\gamma}_{k}] &  \text{if} \; \sigma \leq \bar{\gamma}_{k}  \\
                \bar{\gamma}_{k}  & \text{otherwise,} 
                \end{cases}
                \end{equation}
               where $\bar{\gamma}_{k} > 0$ is the root of the following equation in $\gamma$:
                \end{itemize}
     %           \end{enumerate}
                \vspace{-0.2cm}
                 \begin{align}
                \label{eqt:stepsize}
                       b(x_{k}) d(x_{k})^{\theta-2} \gamma^{\theta}  +   c(x_{k})  d(x_{k})^{\beta-2} \gamma^{\beta} + \left(L_B^2+a(x_{k})\right)\gamma^2   =   \frac{\alpha_{k}}{2}.
                \end{align}
			  }}
	\end{center}

\medskip 

    \begin{remark}
    Stepsize choice  \eqref{eq:gamma} could be just considered  $\gamma_{k} = \bar{\gamma}_{k}$. However, the positive parameter $\sigma$  allows us to choose a larger range of stepsizes.  
    \end{remark}

    \noindent Note that equation \eqref{eqt:stepsize}, which is a polynomial equation when $\mu$, $\theta$, and $\beta$ are integers, is well defined, i.e., there exists  $\bar\gamma_{k}>0$ satisfying equation \eqref{eqt:stepsize}. Indeed, define
    \vspace{-0.5cm}
    \begin{multline*}
    (\forall \gamma \in [0,+\infty))\quad 
     h(\gamma) =  2 b(x_{k}) d(x_{k})^{\theta-2} \gamma^{\theta}  +   2c(x_{k}) d(x_{k})^{\beta-2} \gamma^{\beta} + 2(a(x_{k})+L_B^2)\gamma^2  - \alpha_{k}, \vspace{-0.3cm}
    \end{multline*}

    \vspace{-0.5cm}
    \noindent and $w_{k} = \sqrt{\alpha_{k}}/L_B$. Note that, we have $h(w_{k}) \geq \alpha_{k} > 0$ and $h(0) < 0$. Since $h$ is continuous on $[0,w_{k}]$, there exists $\bar \gamma_{k} \in ]0,w_{k}[$ such that 
    $h(\bar\gamma_{k})=0$. Moreover, since $h'(\gamma) \geq 4L_B^2\gamma >0$ for every $\gamma \in ]0,+\infty[$, then $h$ is strictly increasing  on $]0,+\infty[$. Hence, there exists exactly one $\bar\gamma_{k}>0$ such that the equality in \eqref{eqt:stepsize} is satisfied and  $h(\gamma_{k}) \leq 0 = h(\bar\gamma_{k})$ for $\gamma_{k}$ defined in \eqref{eq:gamma}.

	\begin{lemma}
        \label{lem1}
		Let Assumption \ref{ass1} hold  with $\mu = 2$.
  %, for every $z_{1} \in \mathbb{H}$. 
  Let $k\ge 0$ and  let $\gamma_{k}$ be given by
  \eqref{eq:gamma}.
  %the unique solution to \eqref{eqt:stepsize}.  
  Then, inequality \eqref{eq:12} is satisfied and
        \vspace{-0.2cm}
		\begin{align}
			\gamma_{k} < \eta := \sqrt{\dfrac{\alpha_{\max}}{2  L_B^2  }}. \label{eq:19}
        \vspace{-0.2cm}
		\end{align}
	\end{lemma}

\begin{proof}{Proof}
		\noindent From basic properties of the norm, 
        \vspace{-0.7cm}
		\begin{align*}
			\gamma_{k}^2 \|Ax_{k} + Bx_{k}- Ap_{k}  - Bp_{k} \|^2 \leq 2\gamma_{k}^2 \|Bx_{k} - Bp_{k}\|^2 + 2\gamma_{k}^2 \|Ax_{k} - Ap_{k}\|^2.
		\end{align*}

        \vspace{-0.5cm}
		\noindent Using  the Lipschitz continuity of operator $B$ on $\operatorname{dom}C$ and Assumption \ref{ass1}.\ref{ass1vi} for $\mu \in ]0,+\infty[$, we get 
        \vspace{-0.5cm}
		\begin{align}
			&\gamma_{k}^2 \|Ax_{k}+Bx_{k} - Ap_{k}-Bp_{k}\|^2 \nonumber \\ & \leq 2\gamma_{k}^2 L_B^2\|x_{k}  - p_{k} \|^{2} 
             + 2\gamma_{k}^2 a(x_{k}) \|x_{k}  - p_{k} \|^{\mu} + 2\gamma_{k}^2 b(x_{k}) \|x_{k}  - p_{k} \|^{\theta}      + 2\gamma_{k}^2 c(x_{k}) \|x_{k}  - p_{k}  \|^{\beta} \nonumber \\
			&= 2\gamma_{k}^2 \Big(L_B^2 + a(x_{k}) \|x_{k}  - p_{k} \|^{\mu-2}  + b(x_{k}) \|x_{k}  - p_{k} \|^{\theta-2}     +   c(x_{k})  \|x_{k}  - p_{k} \|^{\beta-2}\Big)  \|x_{k}  - p_{k}  \|^2. \label{eq:18}
            \vspace{-0.2cm}
		\end{align}

        \vspace{-0.5cm}
		\noindent Using  \eqref{ass:resC} with $q = x_{k}$, $w = \hat{x}_{k-1}$, $u = Ax_{k} + Bx_{k}$, and $\gamma=\gamma_{k}$, 
        \vspace{-0.5cm}
		\begin{align}
			\|x_{k}  - p_{k} \|  &= \|\text{proj}_{\text{dom}C}(\hat{x}_{k-1}) - J_{\gamma_{k} C} z_{k}\|  \overset{\eqref{ass:resC}}{\leq} \gamma_{k}(\zeta \|Ax_k + Bx_k\| + \tau) = \gamma_{k} d(x_{k}). \label{eq:21}
        \vspace{-0.2cm}
		\end{align}
        
        \vspace{-0.5cm}
		\noindent From \eqref{eqt:stepsize}, \eqref{eq:18}, \eqref{eq:21}, and the fact that $h(\gamma_{k}) \leq  h(\bar\gamma_{k})$, for $\mu=2$, we deduce that
        \vspace{-0.5cm}
		\begin{align}
			&\gamma_{k}^2 \|Ax_{k} + Bx_{k}- Ap_{k}  - Bp_{k} \|^2 \nonumber \\
            & {\overset{\eqref{eq:18},\eqref{eq:21}}{\leq} \; }2(L_B^2\gamma_{k}^2 + a(x_{k})  \gamma_{k}^{2}   + b(x_{k}) d(x_{k})^{\theta-2} \gamma_{k}^{\theta}  +   c(x_{k})  d(x_{k})^{\beta-2} \gamma_{k}^{\beta})\|x_{k}  - p_{k}  \|^2  \nonumber\\
            &\leq 2\left( (L_B^2 + a(x_{k}))  \bar{\gamma}_{k}^{2} + b(x_{k}) d(x_{k})^{\theta-2} \bar{\gamma}_{k}^{\theta}  +   c(x_{k})  d(x_{k})^{\beta-2} \bar{\gamma}_{k}^{\beta}\right)\|x_{k}  - p_{k}  \|^2 \nonumber \\
            & {\overset{\eqref{eqt:stepsize}}{=}} \; \alpha_{k} \|x_{k} - p_{k}\|^2.  \label{eq:8} 
		\end{align}

        \vspace{-0.5cm}
		\noindent From the above inequality, the first statement holds. Moreover, from \eqref{eqt:stepsize}, since $a(x_{k})$, $b(x_{k})$, $c(x_{k})$, and $d(x_{k})$ are nonnegative for every $k\geq0$, we have
           $ 2L_B^2\bar{\gamma}_{k}^2 - \alpha_{k} \leq 0$. Since $\alpha_{k} < \alpha_{\max}$, for every $k \geq 0$, inequality \eqref{eq:19} holds. \Halmos
	\end{proof}
 
\noindent From previous examples, one can see that Stepsize Choice 1 requires the computation of a positive root of a second-order polynomial equation for quadratically constrained quadratic programs in \eqref{eq:qcqp}, while for quadratic over linear fractional programs \eqref{QuadFracPro}, one needs to compute the positive  root of  a third-order equation. More explicitly: \\
(i) If we consider the quadratically constrained quadratic program  \eqref{eq:qcqp}, then  the operator $A$ defined in \eqref{opA:QCQP} for problem \eqref{eq:qcqp} satisfies \eqref{eq:34} where, for every $i\in \{1,\ldots,m\}$ $g_i\colon x \mapsto \frac{1}{2}x^\top Q_{i}x + l^\top_{i}x -  r_{i}$. Hence, equation \eqref{eqt:stepsize} becomes: 
\vspace{-0.5cm}
\begin{align}
    b\,d(x_{k})^{2} \gamma^{4} + (\|Q_{0}\|^2 +  a(x_{k}))\gamma^2   -  \frac{\alpha_{k}}{2} =0, 
    \label{eq:stepQCQP}
\end{align}

\vspace{-0.5cm}
\noindent with the functions $a$ and $b$ given in \eqref{eq:34}.  Defining, $\hat\gamma = \gamma^2$ we have 
\vspace{-0.5cm}
\begin{align*}
    c\,d(x_{k})^{2} \hat\gamma^{2} + (\|Q_{0}\|^2 +  b(x_{k}))\hat\gamma   -  \frac{\alpha_{k}}{2} =0.
   \end{align*}

   \vspace{-0.5cm}
   \noindent Note that, the positive root of the second order equation is  
   \begin{equation*}
        \hat\gamma_{k} = \dfrac{- (\|Q_{0}\|^2 +  b(x_{k})) + \sqrt{(\|Q_{0}\|^2 +  b(x_{k}))^2 + 2\alpha_{k} c\,d(x_{k})^{2}}}{2c\,d(x_{k})^{2}}. 
    \end{equation*}
    \noindent  Then, using $\hat\gamma_{k} = \gamma^2_{k}$, we obtain that the positive root of \eqref{eq:stepQCQP} is
 \begin{align*}
     \bar{\gamma}_k = \left(\dfrac{\sqrt{(\|Q_{0}\|^2 +  b(x_{k}))^2 + 2 c\,d(x_{k})^{2} \alpha_{k} }-(\|Q_{0}\|^2 +  b(x_{k}))}{2c\, d(x_{k})^{2}}\right)^{1/2}.
 \end{align*} 
(ii) For the quadratic  fractional program \eqref{QuadFracPro}, equation \eqref{eqt:stepsize} becomes 
\vspace{-0.5cm}
\begin{align*}
c(x_{k}) d(x_{k})^{4} \gamma^{6} +  b(x_{k}) d(x_{k})^{2} \gamma^{4}  +   (a(x_{k}) + L_{B}^2) \gamma^2   -  \frac{\alpha_{k}}{2} = 0,
\end{align*}

\vspace{-0.5cm}
\noindent where functions $a$, $b$, and $c$  are given in \eqref{eq:42}. Setting $\eta = \gamma^2$, we obtain a  cubic equation  with a positive root $\eta_{k}$,  and then $\bar \gamma_{k} = \sqrt{\eta_{k}}$.

%%%%%%%%%%%%%%%%%%%%%

\subsubsection{Convergence results under  pseudo-monotonicity} 
	
	\noindent Next, we show the asymptotic convergence of the sequences $(x_{k})_{k\in\mathbb{N}}$ and $(p_{k})_{k\in\mathbb{N}}$ generated by AFBF
    Algorithm  when $\mu = 2$ and the stepsize is computed according to  \eqref{eq:gamma}. 
The following sequence will play a key role in our convergence analysis:
    \vspace{-0.2cm}
     \begin{equation}
(\forall k \in \mathbb{N}) \quad    
u_k =  \gamma_k^{-1} (z_{k}  - p_{k}) + Ap_{k} + Bp_{k} \in Ap_{k} + Bp_{k} + Cp_{k}. \label{eq:15}
    \vspace{-0.1cm}
    \end{equation}

	\begin{theorem}
    \label{theo:1}
		Suppose that $\operatorname{zer}(A+B+C)\neq \varnothing$ and
  Assumption \ref{ass1} holds  with $\mu = 2$. Let
  $(x_{k})_{k\in\mathbb{N}}$, $(z_{k})_{k\in\mathbb{N}}$ $(p_{k})_{k\in\mathbb{N}}$, and $(q_{k})_{k\in\mathbb{N}}$ be
  sequences  generated by AFBF algorithm with stepsizes  $(\gamma_{k})_{k\in \mathbb{N}}$  given by \eqref{eq:gamma}. Then, the following hold: 
  \begin{enumerate}[i)]
        \item $(x_k)_{k\in \mathbb{N}}$ is a Fej\`er monotone sequence with respect to $\operatorname{zer}(A+B+C)$;
		\item $\sum_{k=0}^{+\infty} \|x_{k} - p_{k}\|^2  < +\infty$ and $\sum_{k=0}^{+\infty} \|z_{k} - q_{k}\|^2  < +\infty$;
  %, with $q_{k} = p_{k} -\gamma_{k}\left( Ap_{k} + Bp_{k}\right) $ \\
		\item  there exists $\bar{z} \in \text{zer}(A+B+C)$ such that $x_{k} \to \bar{z}$, $p_{k} \to \bar{z}$,  and $u_{k} \to 0$. 
  \end{enumerate}
	\end{theorem}
	
	\begin{proof}{Proof}
 %\begin{enumerate}[i)]
%\item 
 i) Let $k\in \mathbb{N}$ and let $\bar{z} \in \text{zer}(A+B+C)$. Then,
    \vspace{-0.5cm}
		\begin{align*}
			&\|x_{k} - \bar{z}\|^2 = \|x_{k} - p_{k} + p_{k} - \bar{z}\|^2 \\
			& = \|x_{k} - p_{k}\|^2 + 2 \langle x_{k} - p_{k}, p_{k} - \bar{z} \rangle + \| p_{k}  - \hat{x}_{k} + \hat{x}_{k} - \bar{z}\|^2 \\
			& = \|x_{k} - p_{k}\|^2  + \|p_{k}  - \hat{x}_{k}\|^2 + \|\hat{x}_{k} - \bar{z}\|^2 + 2 \langle x_{k} - p_{k}, p_{k} - \bar{z} \rangle  + 2 \langle p_{k} - \hat{x}_{k}, \hat{x}_{k} - \bar{z} \rangle \\
			& = \|x_{k} - p_{k}\|^2  - \|p_{k}  - \hat{x}_{k}\|^2 + \|\hat{x}_{k} - \bar{z}\|^2 + 2 \langle x_{k} - p_{k}, p_{k} - \bar{z} \rangle  + 2 \langle p_{k} - \hat{x}_{k}, p_{k} - \bar{z} \rangle \\
			& = \|x_{k} - p_{k}\|^2  - \|p_{k}  - \hat{x}_{k}\|^2 + \|\hat{x}_{k} - \bar{z}\|^2  + 2 \langle x_{k} - \hat{x}_{k}, p_{k} - \bar{z} \rangle. \vspace{-0.2cm}
		\end{align*}

        \vspace{-0.5cm}
		\noindent Moreover, using $p_{k}  - \hat{x}_{k} = \gamma_{k} \left( Ap_{k} + Bp_{k} - Ax_{k} - Bx_{k}\right) $, we deduce that
        \vspace{-0.1cm}
		\begin{equation}
			\|x_{k} - \bar{z}\|^2 
			= \|x_{k} - p_{k}\|^2  - \gamma_{k}^2\| Ap_{k} + Bp_{k} - Ax_{k} - Bx_{k} \|^2 + \|\hat{x}_{k}- \bar{z}\|^2 + 2 \langle x_{k} - \hat{x}_{k}, p_{k} - \bar{z} \rangle.\label{eq:40}
        \vspace{-0.1cm}
		\end{equation}
        
		\noindent Note that $\bar{z} \in \text{dom}C$. Using the nonexpansiveness of the projection, \eqref{eq:40} yields
        \vspace{-0.5cm}
		\begin{align}
			\|x_{k+1} - \bar{z}\|^2 &\leq \|\hat{x}_{k}- \bar{z}\|^2  \nonumber \\ & \overset{\eqref{eq:40}}{ =} 	\|x_{k} - \bar{z}\|^2 
			- \|x_{k} - p_{k}\|^2  + \gamma_{k}^2\|Ap_{k} + Bp_{k}  - Ax_{k} - Bx_{k}\|^2 - 2 \langle x_{k} - \hat{x}_{k}, p_{k} - \bar{z}\rangle.\label{eq:13} \vspace{-0.2cm}
		\end{align}
        
        \vspace{-0.5cm}
        \noindent  We deduce from Lemma \ref{lem1} that
          \vspace{-0.1cm}
		\begin{equation}
			\|x_{k+1} - \bar{z}\|^2 
			\leq \|x_{k}- \bar{z}\|^2  - (1-\alpha_k)
            \|x_{k} - p_{k}\|^2- 2 \langle x_{k} - \hat{x}_{k}, p_{k} - \bar{z}\rangle.\label{eq:4}
            \vspace{-0.1cm}
		\end{equation}
		\noindent Since $z_k \in (\operatorname{Id} + \gamma_{k}C)p_{k}$, the inclusion relation in
        \eqref{eq:15} holds and
        \vspace{-0.1cm}
		\begin{equation}\label{e:defuk}
			x_{k} -\hat{x}_{k} = \gamma_k u_k.
        \vspace{-0.1cm}
		\end{equation}
  %where $u_{k}$ is defined in \eqref{eq:15}.
    \noindent Since $A+B+C$ is pseudo-monotone
  and $\bar{z}$ is a zero of $A+B+C$, we obtain:
  \vspace{-0.2cm}
		\begin{equation}
			\langle u_k, p_{k} - \bar{z} \rangle \geq 0.\label{e:pseudnormconv1}
        \vspace{-0.1cm}
		\end{equation}

		\noindent Using the last inequality with \eqref{e:defuk} and  $\alpha_{k} \leq \alpha_{\max}<1$, it follows from \eqref{eq:4} that
  \vspace{-0.1cm}
		\begin{equation}
			\|x_{k+1} - \bar{z}\|^2 \leq \|x_{k} - \bar{z}\|^2  -   (1-\alpha_{\max}) \|x_{k} - p_{k}\|^2 \leq \|x_{k} - \bar{z}\|^2. \label{e:17} 
    \vspace{-0.1cm}
		\end{equation}
  This shows that $(x_k)_{k\in \mathbb{N}}$ is a Fej\`er monotone sequence w.r.t.  $\text{zer}(A+B+C)$.\\
  
%\item
\noindent ii) Since $(x_k)_{k\in \mathbb{N}}$ is a Fej\`er monotone  sequence, then it is bounded and 
  \eqref{e:17} yields
    \vspace{-0.2cm}
		\begin{equation}
			(1-\alpha_{\max}) \sum_{j=0}^{k} \|x_{j} - p_{j}\|^2 \leq \|x_{0} - \bar{z}\|^2 < + \infty. \label{eq:10}
        \vspace{-0.2cm}
		\end{equation}
		
		\noindent It follows that $(p_{k})_{k\in \mathbb{N}}$ is also bounded. 
        In addition, by using Steps 2.(b) and 2.(e) of  AFBF algorithm, Lemma \ref{lem1}, Cauchy-Schwarz inequality, and the fact that $\alpha_{k}\leq \alpha_{\max}$, we further get
        \vspace{-0.5cm}
		\begin{align}
			\|z_{k} - q_{k}\|^2 \;\; & {\overset{2.(e),\eqref{e:defuk}}{=}} \;\; \gamma_k^2\|u_k\|^2 \;\; {\overset{2.(b),\eqref{eq:15}}{=}} \;\; \|x_{k} - p_{k} + \gamma_{k}(Ap_{k}  + Bp_{k}  - Ax_{k} - Bx_{k} )\|^2  \nonumber \\
            & \quad =\|x_{k} - p_{k}\|^2
            +2 \gamma_k \langle x_{k} - p_{k},
            Ap_{k}  + Bp_{k}  - Ax_{k} - Bx_{k}\rangle\nonumber\\
            &\qquad + \gamma_k^2 \| Ap_{k}  + Bp_{k}  - Ax_{k} - Bx_{k}\|^2\nonumber\\
            & \quad\le \|x_{k} - p_{k}\|^2
            +2 \gamma_k \|x_{k} - p_{k}\|
            \|Ap_{k}  + Bp_{k}  - Ax_{k} - Bx_{k}\|\nonumber\\
            &\qquad + \gamma_k^2 \| Ap_{k}  + Bp_{k}  - Ax_{k} - Bx_{k}\|^2\nonumber\\
            & \quad \le (1+\sqrt{\alpha}_k)^2 \|x_{k} - p_{k}\|^2.
            \label{eq:9bis}
            \vspace{-0.2cm}
		\end{align}

		\vspace{-0.5cm}
		\noindent As a consequence of \eqref{eq:10} and the boundedness of $(\alpha_k)_{k\in \mathbb{N}}$,   $\sum_{k=0}^{+\infty} \|z_{k} - q_{k}\|^2  < + \infty$. \\
%\item   
\noindent iii) Let $u_k$ be defined by  \eqref{eq:15}.
        According to \eqref{eq:9bis},
 since $\alpha_k \in ]0,1[$, 
        \vspace{-0.1cm}
		\begin{equation}
			\|u_{k}\| 
			\leq \left( \gamma_{k}^{-1} \right) (1+\sqrt{\alpha_k})\|x_{k} - p_{k}\| \leq 2 \gamma_{k}^{-1} \|x_{k} - p_{k}\|.  \label{eq:14}
    \vspace{-0.1cm}
		\end{equation}
		
	\noindent On other hand, from the definition of $\gamma_{k}$ 
 and Lemma \ref{lem1}, it follows that
 \vspace{-0.5cm}
    \begin{align}
	&\alpha_{\min} \leq \alpha_{k} \label{eq:gammahat} \\
    &= 2\bar\gamma_{k}^2 \left(L_B^2 + a(x_{k})  + b(x_{k}) d(x_{k})^{\theta-2} \bar\gamma_{k}^{\theta-2}  +   c(x_{k})  d(x_{k})^{\beta-2}\bar\gamma_{k}^{\beta-2} \right) \nonumber \\
    &{\overset{\eqref{eq:19}}{\leq}} 2\bar\gamma_{k}^2 \left(L_B^2 + a(x_{k})  + b(x_{k}) d(x_{k})^{\theta-2} \eta^{\theta-2}  +   c(x_{k})  d(x_{k})^{\beta-2}\eta^{\beta-2} \right).  
    \nonumber
    \vspace{-0.1cm}
    \end{align}

    \vspace{-0.5cm}
    \noindent Since $A$, $B$, $a$, $b$, and $c$ are continuous 
    on $\operatorname{dom} C$ and, $(x_{k})_{k\in \mathbb{N}}$ and $(\gamma_{k})_{k\in \mathbb{N}}$ are bounded, then $(d_k)_{k\in \mathbb{N}}$ is bounded and
    there exist $(R_{1},R_{2},R_{3}) \in ]0,+\infty[^3$ such that 
    \vspace{-0.1cm}
    \begin{equation*}
	a(x_{k})  \leq R_{1}, \quad b(x_{k})d(x_{k})^{\theta-2}  \leq R_{2}, \quad \text{and} \quad c(x_{k})d(x_{k})^{\beta-2}  \leq R_{3}.
    \vspace{-0.1cm}
    \end{equation*}

    \noindent This allows us to lower-bound $\gamma_k$ as follows:
    \vspace{-0.1cm}
    \begin{equation}
	\gamma_{k} \geq \gamma_{\min} := \min \left\lbrace \sigma , \sqrt{\dfrac{\alpha_{\min}}{2\left( L_{B}^2 +R_1   + R_2 \eta^{\theta-2}   + R_{3} \eta^{\beta-2}\right)} }\right\rbrace   . \label{eq:5}
    \vspace{-0.1cm}
    \end{equation}
		
	\noindent Hence, from \eqref{eq:14}, we deduce that
    \vspace{-0.2cm}
	\begin{equation}
			\|u_{k}\| \leq 2 \gamma_{\min}^{-1}      \|x_{k}  - p_{k} \|. \label{eq:20}
   \vspace{-0.2cm}
	\end{equation}
	
	\noindent As \eqref{eq:10} implies that $x_{k}  - p_{k}  \to 0$, we have
		\vspace{-0.1cm}
        \begin{equation}
			u_{k} \to 0. \label{eq:16}
        \vspace{-0.1cm}
		\end{equation}

        \noindent To prove the convergence of $(x_k)_{k\in \mathbb{N}}$, according to the Fej\`er-monotone convergence theorem \cite[Lemma 6]{Bro:67}, applied in our finite dimensional setting, where strong and weak convergences are equivalent, it is sufficient to show that every sequential cluster point of $(x_{k})_{k\in\mathbb{N}}$ is a zero of $A+B+C$.
		Let $w$ be such a sequential cluster point. There thus exists a subsequence $(x_{k_{n}})_{n\in \mathbb{N}}$ of $(x_k)_{k\in \mathbb{N}}$ such that $x_{k_{n}} \to w$. It follows 
        from \eqref{eq:10} and \eqref{eq:16} that
        \vspace{-0.2cm}
		\begin{equation*}
			p_{k_{n}} \to w \quad \text{and} \quad u_{k_{n}} \to 0
        \vspace{-0.2cm}
		\end{equation*}

        \noindent Since $A$ and $B$ are continuous operators on $\operatorname{dom} C$, $u_{k_{n}} - Ap_{k_{n}} - B p_{k_{n}} \to -Aw-Bw$.  It follows  from \eqref{eq:15} that $(p_{k_n}, u_{k_{n}}- Ap_{k_{n}} - B p_{k_{n}})$ lies in $\operatorname{gra}C$.
		\noindent Maximally monotonicity of  $C$ implies that $ (w,-Aw-Bw) \in \operatorname{gra}C$ \cite[Proposition 20.38(iii)]{BauCom:11}. 
        Thus, $w\in \operatorname{zer}(A+B+C)$.
        Hence $x_k\to w$ and, since $x_k-p_k \to 0$, $(p_k)_{k\in \mathbb{N}}$ has the same limit.  This concludes our proof. \Halmos 
	\end{proof}

\noindent In \cite{AlaKimWri:24,CaiZhe:22,Pet:22}, the problem of finding a zero of the sum of two operators $B$ and $C$ is considered when $B$ is Lipschitz continuous, $C$ is maximally monotone, and $B + C$ satisfies the weak Minty condition. Next, we analyze the case when we replace the pseudo-monotonicity assumption with the weak Minty condition. Let us first recall this condition. 
\begin{definition}
\noindent  An operator $T:\mathbb{H} \to 2^{\mathbb{H}}$ satisfies the weak Minty condition on $\mathcal{Z} \subset \mathbb{H}$ if there
exists some $\rho \ge 0$ such that the following holds:
\vspace*{-0.2cm}
\begin{equation}
    \langle \hat{w}, w - z \rangle \geq -\rho \|\hat{w}\|^2 \quad \text{for every} \;\; z\in \mathcal{Z}, \, w\in \mathbb{H}, \;\, \text{and} \;\,  \hat{w}\in Tw. \label{eq:WeakMinty}
    \vspace*{-0.1cm}
\end{equation}
\end{definition}
Note that pseudo-monotone operators (see Definition \ref{def:pseudo-mon}) satisfy  the weak Minty condition on their set of zeros $\mathcal{Z}$ with $\rho = 0$.  Weak Minty condition covers, in particular, minimization problems having star-convex or quasar-convex differentiable  objective functions  \cite{HinSidSoh:20}.

\begin{remark}\ 
\begin{enumerate}[i)]
\item First, one can notice  that our proof works with a  weak Minty type condition, where  $\mathcal{Z} = \operatorname{zer}(A+B+C)$ and $\rho=0$,  instead of  Assumption \ref{ass1}.\ref{ass1iv}. 
Indeed, in the proof of Theorem \ref{theo:1}, the 
pseudo-monotonicity of $A+B+C$ has been used to 
derive inequality \eqref{e:pseudnormconv1},
which can also be derived  from the weak Minty condition with $\rho=0$.
\item 
Second, let us replace the pseudo-monotone condition in Assumption \ref{ass1}.\ref{ass1iv} with the assumption that
%Consider
     $A+B+C$ satisfies the weak Minty condition on $\operatorname{zer}(A+B+C)$ with $\rho>0$ and, additionally,  assume  $\operatorname{dom} C$ bounded. From the continuity of $A$, $B$, $a$, $b$, and $c$, and the boundedness of $\operatorname{dom} C$, there exists $(R_{a},R_{b},R_{c}) \in ]0,+\infty[^3$ such that, for every $z \in \operatorname{dom} C$, 
    \begin{equation}
      a(z) \leq R_{a}, \,\ b(z)d(z) \leq R_{b}, \,\ \text{and} \,\   c(z)d(z) \leq R_{c}. \label{eq:bound}
    \end{equation}
     Then, the results from the last theorem hold as long as the following conditions are satisfied: 
     \vspace*{-0.1cm}
    \begin{equation}
         \rho <  \dfrac{2^{-\frac{3}{2}}\sqrt{\alpha_{\min}} (1-\sqrt{\alpha_{\max}})}{(1+\sqrt{\alpha_{\max}}) \sqrt{L_{B}^2 + R_{a}+ R_{b}\eta^{\theta-2} + R_{c}\eta^{\beta -2}}} \; := \rho_{\max}
         \label{eq:condition},
         \vspace*{-0.1cm}
     \end{equation}
    and either
    \begin{equation}
        \sigma \geq \sqrt{\dfrac{\alpha_{\min}}{2\left( L_{B}^2 + R_a    + R_b \eta^{\theta-2}   + R_{c} \eta^{\beta-2}\right)}}, \,\ \text{or} \,\ (\forall k \geq 0)\;\gamma_{k} = \bar\gamma_{k}.  \label{cond:stepsize}
    \end{equation}
   %  $\gamma_{k} \geq 
   %  \dfrac{2 \rho (1 + \sqrt{\alpha_{k}})}{1-\sqrt{\alpha_{k}}}$.
     %\dfrac{4 \rho (1 + \alpha_{k})}{1-\alpha_{k}}$. 
     Indeed, from \eqref{eq:4}, \eqref{e:defuk} and \eqref{eq:WeakMinty},
        \begin{equation*}
        (\forall k \in \mathbb{N})\quad 
			\|x_{k+1} - \bar{z}\|^2 
			\leq \|x_{k}- \bar{z}\|^2 - (1-\alpha_{k})\|x_{k} - p_{k}\|^2 + 2 \gamma_k 
   \rho \|u_{k}\|^2. 
   \end{equation*}
   
    Using \eqref{eq:14}, we obtain
    \begin{equation}
			\|x_{k+1} \!- \bar{z}\|^2 
			\leq \|x_{k}- \bar{z}\|^2 \!-\! (1-\alpha_{k}  - 2 \gamma_{k}^{-1} \rho (1 \!+\! \sqrt{\alpha_{k}})^2) \|x_{k} \!- p_{k}\|^2 . \label{ineq1}
	\end{equation}
    On other hand, \eqref{eq:gamma}, \eqref{eq:gammahat},
    \eqref{eq:bound}, \eqref{eq:condition}, and \eqref{cond:stepsize} yield
    \vspace{-0.5cm}
    \begin{align*}
	\gamma_{k} \geq \min \left\lbrace \sigma , \bar{\gamma_{k}} \right\rbrace  &\geq \sqrt{\dfrac{\alpha_{\min}}{2\left( L_{B}^2+ R_a + R_b \eta^{\theta-2}   + R_{c} \eta^{\beta-2}\right)} } \\ &\geq \dfrac{2 \rho_{\max} (1+\sqrt{\alpha_{\max}})}{1 - \sqrt{\alpha_{\max}}}  \geq \dfrac{2 \rho_{\max} (1+\sqrt{\alpha_{k}})}{1 - \sqrt{\alpha_{k}}}.
    \end{align*}

  \vspace{-0.5cm}  
  \noindent Hence, it follows that 
  \vspace*{-0.1cm}
  \begin{equation*}
    1-\alpha_{k}  - 2 \gamma_{k}^{-1} \rho (1+\sqrt{\alpha_{k}})^2  \geq \left(1 - \dfrac{\rho}{\rho_{\max}}\right) (1-\alpha_{k}) \geq \left(1 - \dfrac{\rho}{\rho_{\max}}\right) (1-\alpha_{\min}).
    \vspace*{-0.2cm}
\end{equation*}
The inequality above and \eqref{ineq1} show that $(x_k)_{k\in \mathbb{N}}$ is a Fej\`er monotone sequence with respect to $\text{zer}(A+B+C)$ and
$\sum_{k=0}^{+\infty} \|x_k-p_k\|^2<+\infty$.
By proceeding similarly to the proof of Theorem \ref{theo:1}, the convergence of $(x_k)_{k\in \mathbb{N}}$ to a
zero of $A+B+C$ can be proved. 
\end{enumerate}
\end{remark}

\noindent Now, we show a sublinear convergence rate result for the iterates of AFBF algorithm. 
    \begin{theorem}
    \label{theo:2}
    Under the same assumptions as in Theorem \ref{theo:1},
%		Let Assumption \ref{ass1} hold  with $a(z_{1})=0$ for all $z_{1} \in H$,  the sequence $(x_{k})_{k\in\mathbb{N}}$, $(z_{k})_{k\in\mathbb{N}}$ and $(p_{k})_{k\in\mathbb{N}}$ be generated by AFBF Algorithm and  the stepsize be generated as in \eqref{eqt:stepsize}. Then, 
the following hold: for every $k_0\in \mathbb{N}$
and $ k\in \mathbb{N}^*$,
\vspace{-0.5cm}
		\begin{align*}
 {\gamma_{\min} 
         %(1+\sqrt{\alpha_{\max}}) 
         \min_{k_0\le j\le k_0+k-1}\|u_j\|}
         &\le 
\min_{k_0\le j\le k_0+k-1}\|x_j-\hat{x}_j\|\\
&\le (1+\sqrt{\alpha_{\max}})  \min_{k_0\le j\le k_0+k-1}\|x_j-p_j\| \le \frac{\varepsilon_{k_0}}{\sqrt{k}},
%			\min_{0\leq j \leq k-1} \|u_{j}\|^2 \leq \dfrac{1}{k} \left( \dfrac{4}{\gamma_{\min}^2(1-\alpha_{\max})} \right) \|x_{0} - \bar{z}\|^2
		\end{align*}

\vspace{-0.5cm}
 \noindent where $u_k$ is defined in \eqref{eq:15} and  $\varepsilon_{k_0} \to 0$ as $k_0 \to +\infty$.
%        with $u_{j} = \gamma_{j}^{-1} (z_{j} - p_{j}) + Ap_{j} + Bp_{j}  \in Ap_{j} + B p_{j} + Cp_{j}$.
	\end{theorem}
    
    \begin{proof}{Proof}
According to \eqref{e:defuk}, \eqref{eq:14}, and \eqref{eq:5},
\begin{equation}\label{e:majxjxhjxjpj}
(\forall j\in \mathbb{N})\qquad
\gamma_{\min} \|u_j\|\le 
\|x_j-\hat{x}_j\| \le (1+\sqrt{\alpha_{\max}})
\|x_j-p_j\|.
\end{equation}
Let $\overline{z}$ be the limit of $(x_j)_{j\in \mathbb{N}}$.
It follows from \eqref{e:17} that
\vspace*{-0.2cm}
\begin{equation*}
(1-\alpha_{\max}) \sum_{j=k_0}^{k_0+k-1} \|x_j-p_j\|^2
\leq \|x_{k_0}-\overline{z}\|^2,
\vspace*{-0.2cm}
\end{equation*}
which leads to
\vspace*{-0.2cm}
\begin{equation*}
\min_{k_0\leq j \leq k_0+k-1} \|x_j-p_j\|^2
\le \frac{1}{(1-\alpha_{\max})k}\, \|x_{k_0}-\overline{z}\|^2.
\end{equation*}
The result follows from the latter equation and \eqref{e:majxjxhjxjpj}, by setting
\vspace*{-0.1cm}
\begin{equation*}
\varepsilon_{k_0} = \frac{1}{\sqrt{1-\alpha_{\max}}} \|x_{k_0}-\overline{z}\|.
\end{equation*} 
Hence, the statement follows. \Halmos 
    \end{proof}

\noindent Note that convergence results in Theorems  \ref{theo:1} and \ref{theo:2} are consistent with those obtained in the literature on (non)monotone inclusion problems  \cite{BriDav:18,ComPes:12,DavWot:17,JiaVan:22,Tse:00}. 

%%%%%%%%%%%%%%%%%%%%%%%%%%%%%%%%%%%%%%%%%%%%

\subsubsection{Convergence results under uniform pseudo-monotonicity}
   \noindent In this section, we refine our convergence results when the operator $A+B+C$ is uniformly pseudo-monotone. Next, we present the definition of a uniformly monotone/pseudo-monotone operator. 

    \begin{definition}
    Let $T\colon \mathbb{H}\to 2^{\mathbb{H}}$.
    \begin{enumerate}[i)]
    \item $T$ is said to be uniformly monotone
    with modulus $q\geq1$ if there exists a constant $\nu >0$ such that, 
    for every $(x,y)\in \mathbb{H}^2$ and $(\hat{x},\hat{y})\in Tx \times Ty$,
    \vspace*{-0.1cm}
        \begin{equation*}
            \langle \hat{x} - \hat{y}, x-y\rangle \geq \dfrac{\nu}{2} \|x-y\|^q.
         \vspace*{-0.1cm}   
        \end{equation*}  
    \item $T$ is said to be uniformly pseudo-monotone with modulus $q\geq1$ if there exists a constant $\nu>0$ such that,
    for every $(x,y) \in \mathbb{H}$ and $(\hat{x},\hat{y}) \in  Tx\times Ty$,
    \vspace*{-0.1cm}
    \begin{equation*}
        \langle \hat{x}, y-x \rangle \geq 0 \quad \Longrightarrow \langle \hat{y}, y-x \rangle \geq \dfrac{\nu}{2} \|x-y\|^q .
     \vspace*{-0.1cm}   
    \end{equation*}
    \end{enumerate}
\end{definition}

\noindent When $q=2$ in the definition above, we say that operator $T$ is strongly monotone / pseudo-monotone. Note that, if $T$ is uniformly monotone, then $T$ is also uniformly pseudo-monotone.

\begin{example}
    Consider a proper uniformly convex function $f\colon \mathbb{R}^n \to ]-\infty,+\infty[$. The subdifferential $\partial f$ of $f$ is uniformly monotone \cite[Example 22.5]{BauCom:11}
\end{example}

  \noindent  Below we give an example of a strongly pseudo-monotone map that is not monotone.
  \begin{example}
      Consider the unit ball $U = \{x \in \mathbb{R}^{n}\mid \|x\| \leq 1 \}$ and the map $F\colon U\setminus\{0\} \to \mathbb{R}^{n}$ such that
      \vspace*{-0.1cm}
    \begin{equation*}
    (\forall x \in U\setminus\{0\})\quad 
	F(x) = \left(  \dfrac{2}{\|x\|} - 1 \right) x. 
    \vspace*{-0.1cm}
    \end{equation*}

	\noindent Note that $F$ is not monotone on $U\setminus\{0\}$. For example, setting $y=(1,0,\ldots,0)$ and $w = (1/2,0,\ldots,0)$ yields
    \vspace*{-0.2cm}
	\begin{equation*}
		\langle F(y) - F(w), y - w \rangle = - \dfrac{1}{4}.
    \vspace*{-0.1cm}
	\end{equation*}
	 However, $F$ is strongly pseudo-monotone on $U\setminus\{0\}$. Indeed, for every $(x,y) \in (U\setminus\{0\})^2$, if $\langle F(x), y-x\rangle \geq 0$, then $\langle x, y -x \rangle \geq 0$,
  and consequently:
  \vspace*{-0.1cm}
	\begin{equation*}
		\langle F(y),y-x\rangle = (2\|y\|^{-1}  -1 )\langle y, y-x \rangle   \geq (2\|y\|^{-1} -1 ) \langle y - x, y-x \rangle \geq \|y-x\|^{2}.
    \vspace*{-0.1cm}
	\end{equation*}
  \end{example}
  
 \noindent Next, considering operators $A$, $B$, $C$ satisfying Assumption \ref{ass1} with $\mu = 2$ and stepsizes $(\gamma_{k})_{k\in \mathbb{N}}$ computed as in \eqref{eq:gamma}, we derive linear convergence rates when $A+B+C$ is uniformly pseudo-monotone with modulus $q \in [1,2]$, and sublinear rates when $q>2$. 
 %First, let us define constants that will be used to quantify the convergence rate. 
 %From Theorem \ref{theo:1}, the sequence $(p_{k})_{k\in\mathbb{N}}$ generated by AFBF Algorithm is convergent. Hence, for some $\bar{z} \in \text{zer}(A+B+C)$, there exists $R>0$ such that 
% \begin{align}
% (\forall k \in \mathbb{R})\quad 
% \|p_{k} - \bar{z}\| \leq R. \label{def:R}
%  \end{align}
%Moreover, define
%\begin{align}
%    r = \min \{ 1-\alpha_{\max}, \gamma_{\min} \nu R^{q-2} \} < 1.
%    \frac{1}{2^{q-1}} \min \left\{ \frac{ (1-\alpha_{\max}) }{R^{q-2}}, \gamma_{\min} \nu \right\}.
%\label{def:r}
%\end{align}
 
	\begin{theorem}
    \label{theo3}
		Suppose that Assumption \ref{ass1} holds with $\mu = 2$. Let $(x_{k})_{k\in\mathbb{N}}$, %$(z_{k})_{k\in\mathbb{N}}$, 
  and $(p_{k})_{k\in\mathbb{N}}$ be sequences generated by AFBF algorithm with stepsizes  $(\gamma_{k})_{k\in \mathbb{N}}$  given by \eqref{eq:gamma}. Assume that $A+B+C$ is uniformly pseudo-monotone with modulus $q\ge 1$ and constant $\nu>0$. 
Then, 
  for some $\bar{z} \in \text{zer}(A+B+C)$ and constants
  \vspace*{-0.1cm}
    \begin{equation}\label{def:R}
     R =\sup_{k\in \mathbb{N}}\|p_{k} - \bar{z}\|<+\infty \quad \text{and} \quad r = \min \{ 1-\alpha_{\max}, \gamma_{\min} \nu R^{q-2} \} < 1, 
    \vspace*{-0.1cm}
\end{equation}  
 the following hold:
  \begin{enumerate}[i)]
        \item For $q \in [1,2]$, $x_k$ converges to $\bar z$ linearly:
        \vspace*{-0.2cm}
		\begin{equation}
  (\forall k \in \mathbb{N})\quad 
			\|x_{k} - \bar{z}\| \leq \left(1
			- \frac{r}{2}  \right)^{k/2} \|x_{0} - \bar{z}\|.
        \vspace*{-0.2cm}
		\end{equation}

        \item For $q>2$ and  $\overline{r} = \dfrac{r}{2^{q-1}R^{q-2}}$, $x_k$ converges to $\bar z$ sublinearly: 
        \vspace*{-0.2cm}
        \begin{align*}
        (\forall k \in \mathbb{N})\quad 
		\|x_{k} - \bar{z}\| \leq \dfrac{ \|x_{0} - \bar{z}\| }{\left( \frac{q-2}{2} \overline{r}\|x_{0} - \bar{z}\|^{q-2}     k+1 \right)^{\frac{1}{q-2}}}.
       \vspace*{-0.2cm}
		\end{align*}
  \end{enumerate}

	\end{theorem}

	\begin{proof}{Proof}
 From Theorem \ref{theo:1}, the sequence $(p_{k})_{k\in\mathbb{N}}$ generated by AFBF algorithm is convergent. Hence, for some $\bar{z} \in \text{zer}(A+B+C)$, we have  $R=\sup_{k\in \mathbb{N}}\|p_{k} - \bar{z}\|<+\infty$.
		Since $A+B+C$ is uniformly pseudo-monotone with 
  modulus $q$ and constant $\nu > 0$, 
        \vspace*{-0.1cm}
		\begin{equation*}
        (\forall k \in \mathbb{N})\quad 
			\langle u_k , p_{k} - \bar{z} \rangle \geq \dfrac{\nu}{2} \|p_{k} - \bar{z}\|^q.
        \vspace*{-0.1cm}
		\end{equation*}
		
		\noindent It follows from \eqref{eq:13} and \eqref{e:defuk} that
        \vspace*{-0.1cm}
		\begin{equation*}
			\|x_{k+1} - \bar{z}\|^2 \leq 	\|x_{k} - \bar{z}\|^2 
			- \|x_{k} - p_{k}\|^2  + \gamma_{k}^2\|Ap_{k} + Bp_{k}  - Ax_{k} - Bx_{k}\|^2 - \gamma_{k} \nu  \|p_{k} - \bar{z}\|^q.
        \vspace*{-0.1cm}
		\end{equation*}
		
		\noindent Since $\alpha_{k} \leq \alpha_{\max}$ and $\gamma_{\min} \leq \gamma_{k}$, we deduce from Lemma \ref{lem1} that
        \vspace*{-0.1cm}
		\begin{equation}
			\|x_{k+1} - \bar{z}\|^2  \leq 	\|x_{k} - \bar{z}\|^2 
			- (1-\alpha_{\max}) \|x_{k} - p_{k}\|^2  - \gamma_{\min} \nu  \|p_{k} - \bar{z}\|^q. \label{eq:7}
        \vspace*{-0.1cm}
		\end{equation}

\noindent i)  If $q\in[1,2]$, using the definition of $R$, we get
        \vspace{-0.5cm}
		\begin{align*}
			&(1-\alpha_{\max}) \|x_{k} - p_{k}\|^2  + \gamma_{\min} \nu  \|p_{k} - \bar{z}\|^q \\
            &{\overset{\eqref{def:R}}{\geq}} (1-\alpha_{\max}) \|x_{k} - p_{k}\|^2  + \gamma_{\min} \nu R^{q-2}  \|p_{k} - \bar{z}\|^2 \\
            &\geq \min \{ 1-\alpha_{\max}, \gamma_{\min} \nu R^{q-2} \}  \left(\|x_{k} - p_{k}\|^2  + \|p_{k} - \bar{z}\|^2 \right) {\overset{\eqref{def:R}}{\geq}} \frac{r}{2} \|x_{k} - \bar{z}\|^2.
		\end{align*}

        \vspace{-0.5cm}
		\noindent Combining the two last inequalities we obtain
        
		\begin{equation*}
			\|x_{k+1} - \bar{z}\|^2  \leq 	\left(1
			- \frac{r}{2} \right) \|x_{k} - \bar{z}\|^2.
		\end{equation*}
		
		\noindent 
  Therefore, unrolling the above inequality allows us to prove the first statement.
  
 % following inequality 
%		\begin{align*}
%			\|x_{k+1} - \bar{z}\|^q \leq \left(1
%			- \frac{1}{2}\min \{ (1-\alpha_{\max}), \gamma_{\min} \nu R^{q-2}\} \right)^{\frac{q}{2}} \|x_{k} - \bar{z}\|^q,
%		\end{align*}
%		we get the first statement. \\
\noindent ii) If $q>2$, it follows from \eqref{def:R} that
    \vspace{-0.5cm}
	\begin{align*}
		&(1-\alpha_{\max}) \|x_{k} - p_{k}\|^2  + \gamma_{\min} \nu  \|p_{k} - \bar{z}\|^q \\
		&{\overset{\eqref{def:R}}{\geq}} \dfrac{ (1-\alpha_{\max}) }{R^{q-2}}   \|x_{k} - p_{k}\|^q  + \gamma_{\min} \nu  \|p_{k} - \bar{z}\|^q \\
		&\geq \min \left\{ \dfrac{ (1-\alpha_{\max}) }{R^{q-2}}, \gamma_{\min} \nu \right\}  \left(\|x_{k} - p_{k}\|^q  + \|p_{k} - \bar{z}\|^q \right) \geq \overline{r} \|x_{k} - \bar{z}\|^q.
	\end{align*}

    \vspace{-0.5cm}
	\noindent where in the last inequality we used the fact that $\|a + b \|^q \leq 2^{q-1} \|a\|^q + 
 2^{q-1} \|b\|^q$ for $q \geq 1$. Therefore, using \eqref{eq:7}, we obtain
		$\|x_{k+1} - \bar{z}\|^2  \leq 	\|x_{k} - \bar{z}\|^2 
		- \overline{r} \|x_{k} - \bar{z}\|^q$. 
 Multiplying the inequality above by $\overline{r}^{\frac{2}{q-2}}$, we obtain
	\begin{equation*}
		\overline{r}^{\frac{2}{q-2}}\|x_{k+1} - \bar{z}\|^2  \leq 	\overline{r}^{\frac{2}{q-2}}\|x_{k} - \bar{z}\|^2 
		- \left(\overline{r}^{\frac{2}{q-2}} \|x_{k} - \bar{z}\|^2\right)^{\frac{q}{2}}.
	\end{equation*}

\noindent Applying \cite[Lemma 8]{NecCho:21} (see Appendix) for $\zeta = \frac{q-2}{2}>0$,  we get
\begin{equation*}
   %&\overline{r}^{\frac{2}{q-2}}\|x_{k} - \bar{z}\|^2 \leq \dfrac{ \overline{r}^{\frac{2}{q-2}}\|x_{0} - \bar{z}\|^2 }{\left( \frac{q-2}{2} \left(\overline{r}^{\frac{2}{q-2}}\|x_{0} - \bar{z}\|^2\right)^{\frac{q-2}{2}}     k+1 \right)^{\frac{2}{q-2}}} \\
   %\Leftrightarrow\quad 
    \|x_{k} - \bar{z}\| \leq \dfrac{ \|x_{0} - \bar{z}\| }{\left( \frac{q-2}{2} \overline{r} \|x_{0} - \bar{z}\|^{q-2}k+1 \right)^{\frac{1}{q-2}}}. 
   %\\
   %\Leftrightarrow & \|x_{k} - \bar{z}\|^q \leq \dfrac{ \|x_{0} - \bar{z}\|^q }{\left( \frac{q-2}{2} r\|x_{0} - \bar{z}\|^{q-2}     k+1 \right)^{\frac{q}{q-2}}}.
\end{equation*}
This proves the second statement of the theorem. \Halmos
\end{proof}

\begin{remark}
In Theorem \ref{theo3} , we can replace the assumption of uniform pseudo-monotonicity with the following one: there exists $\nu>0$ and $q\geq 1$, such that, for every $w\in \mathbb{H}$, $\hat{w} \in (A+B+C)w$, and $\bar{z} \in \operatorname{zer}(A+B+C)$, the following inequality holds:
\begin{equation}
    \label{eq:38}
    \langle \hat{w}, w - \bar{z} \rangle \geq \nu \|w-\bar{z}\|^{q}. 
\end{equation}
Proceeding similarly to the proof of Theorem  \ref{theo3}, linear and sublinear rates can be derived under this condition. Condition \eqref{eq:38}, with $q=2$, covers, e.g., minimization problems with strongly star-convex  differentiable objective function or strongly quasi-convex   objective functions  \cite{HinSidSoh:20}. 
\end{remark}

\subsection{Second adaptive choice for the stepsize}

 \noindent   In this section, we present another  possible \textit{adaptive} choice for the stepsize when the operator $A$ satisfies Assumption \ref{ass1}.\ref{ass1vi} with $\mu \in ]0,2[$. Let $\epsilon \in ]0,1[$ be the desired accuracy  for solving problem \eqref{prob}, i.e.,  to obtain $u$ in the range of $A + B + C$ such that $\|u\|\leq \epsilon$. The procedure is described below.

\medskip

\begin{center}
		\noindent\fbox{%
			\parbox{12.7cm}{%
    \small
				\textbf{Stepsize Choice 2}:\\
  %  \begin{enumerate}
                %\item 
                1. Choose $\epsilon \in ]0,1[$, $0<\alpha_{\min}\leq \alpha_{\max} <1$, and $\sigma>0$. \\
                2. For $k \geq 0$ do:
                \begin{itemize}
				\item [(a)] Choose $\alpha_{k}\in [\alpha_{\min},\alpha_{\max}]$ and compute $d(x_{k}) =  \zeta \|Ax_k + Bx_k\| + \tau$.
                \item [(b)]\label{s:s2s2}   Compute $\bar{\gamma}_{k}^{(1)}>0$ as the solution to the equation 
                \begin{align}
                    \label{gammabar1}
                    L_{B}^2 \gamma^2  + b(x_{k}) d(x_{k})^{\theta-2}\gamma^{\theta} + c(x_{k}) d(x_{k})^{\beta-2}\gamma^\beta  + 2^{2-\mu}a(x_{k}) \gamma^{\mu} \epsilon^{\mu-2} = \dfrac{\alpha_{k}}{2}
                \end{align}
                \item [(c)]\label{s:s2s3}   Compute $\bar{\gamma}_{k}^{(2)}>0$ as the solution to the equation 
                \begin{align}
                    L_{B}^2 d(x_{k})^{2-\mu} \gamma^{2} & +   b(x_{k}) d(x_{k})^{\theta-\mu} \gamma^{\theta}  + c(x_{k}) d(x_{k})^{\beta-\mu} \gamma^{\beta}  + a(x_{k})\gamma^{\mu}  = \dfrac{\epsilon^{2-\mu}}{2^{3-\mu}}\alpha_{k} 
                    \label{gammabar2}
                \end{align}
                \item [(d)] Update 
                \begin{equation}
                \bar\gamma_{k} = \min \left\{\bar{\gamma}_{k}^{(1)},\bar{\gamma}_{k}^{(2)}\right\}  \label{eq:stepsize2}
                \end{equation}
                 \item [(e)] Choose $\gamma_{k}$ such that
                \begin{equation} \label{stp2:gamma}
              \gamma_{k} \in   \begin{cases}
                 [\sigma,\bar{\gamma}_{k}] &  \text{if} \; \sigma \leq \bar{\gamma}_{k}  \\
                \bar{\gamma}_{k}  & \text{otherwise.} 
                \end{cases}
                \end{equation}
   \end{itemize}
 %  \end{enumerate}
			  }}
      \vspace*{-0.3cm} 
	\end{center}

\medskip

\noindent Note that $\gamma$ is well-defined in Steps 2.(b) and 2.(c) of 
this second procedure for the choice of the stepsize, i.e., there exist unique $\bar{\gamma}_{k}^{(1)},\bar{\gamma}_{k}^{(2)}$ satisfying \eqref{gammabar1} and \eqref{gammabar2}, respectively. Indeed, consider the functions
\vspace{-0.5cm}
\begin{align*}
    h(\gamma) &= \gamma^2L_{B}^2   + b(x_{k}) d(x_{k})^{\theta-2}\gamma^{\theta} + c(x_{k}) d(x_{k})^{\beta-2}\gamma^\beta  + 2^{2-\mu}a(x_{k}) \gamma^{\mu} \epsilon^{\mu-2} - \dfrac{\alpha_{k}}{2} \\
    r(\gamma) &= L_{B}^2 d(x_{k})^{2-\mu} \gamma^{2} +   b(x_{k}) d(x_{k})^{\theta-\mu} \gamma^{\theta}  + c(x_{k}) d(x_{k})^{\beta-\mu} \gamma^{\beta} + a(x_{k})\gamma^{\mu} - \dfrac{\epsilon^{2-\mu}}{2^{3-\mu}}\alpha_{k}, 
\end{align*}

\vspace{-0.5cm}
\noindent and variables $w_{k} = \dfrac{\sqrt{\alpha_{k}}}{L_{B}}$ 
and $\bar w_{k} = \dfrac{\sqrt{\alpha_{k}}}{L_{B} d(x_{k})^{\frac{2-\mu}{2}}}$. 
Note that $h(0) < 0$ and $h(w_{k})  \geq \alpha_{k}/2 > 0$. Since $h$ is continuous on $[0,w_{k}]$ there exist $\bar{\gamma}_{k}^{(1)} \in ]0,w_{k}[$ such that $h(\bar{\gamma}_{k}^{(1)}) = 0$. Moreover,  since $h'(\gamma)\geq 2\gamma L_{B}^{2} > 0 $ for every $\gamma \in ]0,+\infty[$, then $h$ is strictly increasing in $(0,+\infty)$. Hence, there exists exactly one $\bar{\gamma}_{k}^{(1)} > 0$ such that $h(\bar{\gamma}_{k}^{(1)} ) = 0$. Using the same arguments, we can conclude that $r$ is strictly increasing on $]0,+\infty[$ and there exists only one $\bar{\gamma}_{k}^{(2)} \in ]0,\bar w_{k}[$ such that  $r(\bar{\gamma}_{k}^{(2)})=0$. Since both functions $h$ and $r$ are strictly increasing in $(0,+\infty)$, $h(0) < 0$ and $r(0) < 0$, $\gamma_{k}$ defined in 
\eqref{stp2:gamma} satisfies the following two inequalities: 
\begin{equation}
h(\gamma_k) \le 0 \quad \text{and} \quad r(\gamma_k) \le 0. \label{eq:23}
\end{equation}

\noindent Note that 
\begin{equation}
\label{eq:upperBound}
\bar{\gamma}_{k}^{(1)} \leq \eta \,\ \text{and} \,\    \bar{\gamma}_{k}^{(2)} \leq \bar\eta:= \left( \dfrac{\epsilon^{2-\mu} \alpha_{\max}}{2^{3-\mu} L_{B}^2 \tau^{2-\mu}}\right)^{\frac{1}{2}},
\end{equation}
with $\eta$ defined in \eqref{eq:19}, where in the second inequality we used the fact that $d(x_{k})\geq \tau$. The theorem below provides a  bound on the number of iterations required, for a given $\epsilon > 0$,  to generate $\|u_{k}\| \leq  \epsilon$, with $u_k$ defined in \eqref{eq:15}.

\begin{theorem}\label{theo:convratemule2}
		Let $\epsilon \in ]0,1[$. Suppose that Assumption \ref{ass1} holds with $\mu \in ]0,2[$. Let 
  $(x_{k})_{k\in\mathbb{N}}$ and $(p_{k})_{k\in\mathbb{N}}$ be the sequences generated by AFBF algorithm  with stepsizes $(\gamma_{k})_{k\in\mathbb{N}}$  given by \eqref{stp2:gamma}. Then, for  $u_{k} = \gamma_{k}^{-1} (x_{k} - p_{k}) + Ap_{k} + Bp_{k}  - Ax_{k} - Bx_{k}  \in Ap_{k} + B p_{k} + Cp_{k}$ and $\gamma_{\min}(\epsilon) = \mathcal{O}(\epsilon^{(2-\mu)/\mu})$, performing 
        \begin{equation*}
            K \geq  \dfrac{1}{\epsilon^2} \left( \dfrac{(1+\sqrt{\alpha_{\max}})^2}{\gamma_{\min}^2(\epsilon) (1-\alpha_{\max})} \right) \|x_{0} - \bar{z}\|^2
        \end{equation*}
        iterations ensures that there exists $k \in \{0,\cdots,K-1\}$ such that 
       $\|u_{k}\| \leq  \epsilon $.
\end{theorem}

\begin{proof}{Proof}

\begin{enumerate}[i)]
\item First, consider the case when, for every $k\in \{0,\ldots,K-1\}$,
$\gamma_{k}^{-1}\|x_{k} - p_{k}\|>\epsilon/2$. 
%Using  Lipschitz property  of operator $B$ on $\operatorname{dom}C$, Assumption \ref{ass1}.\ref{ass1vi} and the definition of $d(x_{k})$, 
We deduce from  \eqref{eq:18}, \eqref{eq:21} and \eqref{eq:23}  that
\vspace{-0.5cm}
\begin{align}
&\gamma_{k}^2 \|Ax_{k} + Bx_{k}- Ap_{k}  - Bp_{k} \|^2 \nonumber\\
%&\leq 2\gamma_{k}^2 \|Bx_{k} - Bp_{k}\|^2 + 2\gamma_{k}^2 \|Ax_{k} - Ap_{k}\|^2 \nonumber\\
 %& \leq 2\gamma_{k}^2L_{B}^2\|x_{k} - p_{k}\|^2 +  2\gamma_{k}^2a(x_{k}) \|x_{k} - p_{k}\|^{\mu} \nonumber\\ 
 %& \quad + 2\gamma_{k}^2b(x_{k}) \|x_{k} - p_{k}\|^{\theta} + 2\gamma_{k}^2 c(x_{k}) \|x_{k} - p_{k}\|^{\beta} \nonumber\\
% & = 2\gamma_{k}^2(L_{B}^2 + a(x_{k}) \|x_{k} - p_{k}\|^{\mu-2} + b(x_{k}) \|x_{k} - p_{k}\|^{\theta-2}  \label{eq:39}\\ 
 %& \quad + c(x_{k})\|x_{k} - p_{k}\|^{\beta-2} ) \|x_{k} - p_{k}\|^2 \nonumber \\
 &{\overset{\eqref{eq:18},\eqref{eq:21}}{\leq}} 2\left(\gamma_{k}^2L_{B}^2   + b(x_{k}) d(x_{k})^{\theta-2}\gamma_{k}^{\theta} + c(x_{k}) d(x_{k})^{\beta-2}\gamma_{k}^\beta  + 2^{2-\mu}a(x_{k}) \gamma_{k}^{\mu} \epsilon^{\mu-2} \right) \|x_{k} - p_{k}\|^2 \nonumber \\
 &{\overset{\eqref{eq:23}}{\leq}} \alpha_{k} \|x_{k} - p_{k}\|^2.  \nonumber
 %&= 2\gamma_{k}^\beta (  ( M^2 +  b(x_{k}) d_{k}^{\theta-2} + c(x_{k}) d_{k}^{\beta-2}  )  \gamma_{k}^{2-\mu}  + 2^{2-\mu}a(x_{k}) \epsilon^{\mu-2} ) \|x_{k} - p_{k}\|^2 \nonumber\\
 %&\leq 2\gamma_{k}^\mu \hat{h}_{k} \|x_{k} - p_{k}\|^2 \leq \alpha_{k} \|x_{k} - p_{k}\|^2.  \nonumber
\end{align}

\vspace{-0.5cm}
\noindent Let $\bar{z} \in \text{zer}(A+B+C)$.  Since $\alpha_{k} \leq \alpha_{\max}$, using a similar reasoning as in  \eqref{eq:13}, the inequality \eqref{e:17} also holds when $k\in \{0,\ldots,K-1\}$, for the Stepsize Choice 2.
This implies that 
\vspace{-0.5cm}
\begin{align}
    (1-\alpha_{\max}) \sum_{k=0}^{K-1} \|x_{k} - p_{k}\|^2 \leq \|x_{0} - \bar{z}\|^2,\label{e:sumxpK}\\(\forall k \in \{0,\ldots,K\})\quad 
    \|x_{k}-\bar{z}\| \leq \|x_0 - \bar{z}\|. \nonumber %\label{eq:37} 
\end{align}

\vspace{-0.5cm}
\noindent Let $D$ be the closed ball of center 
$\bar{z}$ and radius  $\|x_0 - \bar{z}\|$.
Since $A$, $B$, $a$, $b$, and $c$ are continuous on $\operatorname{dom} C$, the quantities define below take  finite values:
\vspace{-0.5cm}
\begin{align*}
R_1 = \sup_{x\in D} a(x),
\quad R_2 = \sup_{x\in D} b(x)d(x)^{\theta-2},
\quad R_3 = \sup_{x\in D} c(x)d(x)^{\beta-2}
\\
R_4 = \sup_{x\in D}  d(x)^{2-\mu},\quad
R_5 = \sup_{x\in D} b(x)d(x)^{\theta-\mu},\quad
R_6 = \sup_{x\in D} c(x)d(x)^{\beta-\mu}. 
%a(x_{k})  \leq R_{1}, \quad b(x_{k})d_{k}^{\theta-2}  \leq R_{2},  \quad c(x_{k})d_{k}^{\beta-2}  \leq R_{3}, \\
%    d_{k}^{2-\mu}  \leq R_{4}, \quad b(x_{k})d_{k}^{\theta-\mu}  \leq R_{5}, \quad \text{and} \quad c(x_{k})d_{k}^{\beta-\mu}  \leq R_{6}
    \end{align*}
%for  $k = 0:K-1$. 

\vspace{-0.5cm}
\noindent From \eqref{eq:upperBound}, \eqref{gammabar1} and  \eqref{gammabar2}, one can lower-bound the stepsize as
    \begin{equation*}
    (\forall k\in \{0,\ldots,K-1\})\quad 
        \gamma_{k} \geq  \gamma_{\min}(\epsilon):=\min\{\gamma_{\min}^{(1)}(\epsilon), \gamma_{\min}^{(2)}(\epsilon),\sigma\}, 
    \end{equation*}
    with
    \begin{equation}
	 \gamma_{\min}^{(1)}(\epsilon) :=   \left(\dfrac{\alpha_{\min}}{2\left( L_{B}^2 \eta^{2- \mu} + 2^{2-\mu}R_1 \epsilon^{\mu-2}  + R_2 \eta^{\theta- \mu}   + R_{3} \eta^{\beta- \mu}\right)} \right)^{\frac{1}{\mu}}
    \end{equation}
    \noindent 
    and
    \begin{equation}
	 \gamma_{\min}^{(2)}(\epsilon)  :=   \left(\dfrac{\epsilon^{2-\mu}\alpha_{\min}}{2^{3-\mu}\left(  R_1   + L_{B}^2 \bar\eta^{2-\mu} R_{4} + R_5 \bar\eta^{\theta-\mu}   + R_{6} \bar\eta^{\beta-\mu}\right)} \right)^{\frac{1}{\mu}}.
    \end{equation}

Note that, if $\epsilon$ is sufficiently small, then  $\gamma_{\min}(\epsilon) = \mathcal{O}(\epsilon^{(2-\mu)/\mu})$.  Using \eqref{eq:14}, we finally obtain
\begin{equation*}
(\forall k \in \{0,\ldots,K-1\})\quad
\|u_k\| \le (\gamma_{\min}(\epsilon))^{-1} (1+\sqrt{\alpha}_{\max}) \|x_k-p_k\|,
\end{equation*}
which, by virtue of \eqref{e:sumxpK}, yields
\begin{equation*}
			\min_{0\leq k \leq K-1} \|u_{k}\|^2 \leq \dfrac{1}{K} \left( \dfrac{(1+\sqrt{\alpha_{\max}})^2}{\gamma_{\min}^2(\epsilon)(1-\alpha_{\max})} \right) \|x_{0} - \bar{z}\|^2.
\end{equation*}
%Using \eqref{eq:39} and a similar analysis that yields \eqref{eq:14}, but based on the second choice of the stepsize, together with  \eqref{eq:37} and $\alpha_{k} \leq \alpha_{\max}$, we finally obtain:  
%\begin{align*}
%			\min_{0\leq k \leq K-1} \|u_{k}\|^2 \leq \dfrac{1}{K} \left( \dfrac{(1+\sqrt{\alpha_{\max}})^2}{\gamma_{\min}^2(1-\alpha_{\max})} \right) \|x_{0} - \bar{z}\|^2.
%\end{align*}

\item Second, consider the case when, there exists 
$k \in \{0, \ldots,K-1\}$ such that
$\gamma_{k}^{-1} \|x_{k} - p_{k}\| \leq \epsilon/2$. Let us prove that  $\|Ap_{k} + Bp_{k}  - Ax_{k} - Bx_{k}\| \leq \epsilon/2$. 
Indeed, we deduce from \eqref{eq:18}  that
%using  Lipschitz property of operator $B$ on $\operatorname{dom}C$, Assumption \ref{ass1}.\ref{ass1vi} and the definition of $d_{k}$, we get: 
\vspace{-0.5cm}
\begin{align*}
&\|Ap_{k} + Bp_{k}  - Ax_{k} - Bx_{k}\|^2 \\
 &\leq 2 L_{B}^2\|x_{k} - p_{k}\|^2 + 2 a(x_{k}) \|x_{k} - p_{k}\|^{\mu}+ 2 b(x_{k}) \|x_{k} - p_{k}\|^{\theta} + 2 c(x_{k}) \|x_{k} - p_{k}\|^{\beta} \\
% & = 2 ( L_{B}^2\|x_{k} - p_{k}\|^{2-\mu}  +  b(x_{k}) \|x_{k} - p_{k}\|^{\theta-\mu} + c(x_{k}) \|x_{k} - p_{k}\|^{\beta-\mu} \\
% & \quad  + a(x_{k}) ) \, \|x_{k} - p_{k}\|^{\mu} \\
 &{\overset{\eqref{eq:21}}{\leq}} 2( L_{B}^2 d(x_{k})^{2-\mu} \gamma_{k}^{2-\mu} +   b(x_{k}) d(x_{k})^{\theta-\mu} \gamma_{k}^{\theta-\mu}+ c(x_{k}) d(x_{k})^{\beta-\mu} \gamma_{k}^{\beta-\mu}  + a(x_{k}))\|x_{k} - p_{k}\|^{\mu}\\
&\leq 2( L_{B}^2 d(x_{k})^{2-\mu} \gamma_{k}^{2} +   b(x_{k}) d(x_{k})^{\theta-\mu} \gamma_{k}^{\theta}  + c(x_{k}) d(x_{k})^{\beta-\mu} \gamma_{k}^{\beta} + a(x_{k})\gamma_{k}^{\mu}) \dfrac{\epsilon^{\mu}}{2^{\mu}}  {\overset{\eqref{eq:23}}{\leq}} \; \dfrac{\epsilon^2\alpha_{k}}{4} \leq \dfrac{\epsilon^2}{4}. 
\end{align*}

\vspace{-0.5cm}
\noindent  Hence, from the definition of $u_{k}$, applying the triangle inequality leads to $\|u_{k}\| \leq \epsilon$.  Hence, the statement of the theorem  is proved.   \Halmos 
\end{enumerate}

\end{proof}

\noindent It can be noticed that the literature on convergence rates for the general inclusion problem addressed in this section is scarce. Existing results predominantly focus on the composite problem outlined in Example \ref{ex:3terms}, particularly when $g=0$ and $L=I_n$, spanning both the convex case \cite{Nes:15} and the nonconvex one
\cite{Yas:16}.

%\noindent \red{ION:  For the composite problem given in  Example \ref{ex:3terms}, with $g=0$ and $L=I_n$,  \cite{Nes:15} introduces an universal gradient method  with convergence rate of order $O(\epsilon^{-2/(1+\nu)})$ in function values for the convex (maximally monotone) case , where $\nu$ is the constant from Definition \ref{def:holder}. On the other hand, for  Example \ref{ex:3terms}, with $g=0$ and $h=0$ in the nonconvex (nonmonotone) case, \cite{Yas:16} considers a gradient type method with an adaptive stepsize and  the convergence rate obtained is of order $O(\epsilon^{-\left(1+ \frac{1}{\nu}\right)})$ in the norm of the gradient (when $\nu \to 0$ the rate is very high). The convergence rate  obtained in the last theorem is of order $O(\epsilon^{-2/\nu})$ (recall that $\mu=2 \nu$ in this case), in the norm of the gradient, which is worse than the convergence rate obtained in \cite{Yas:16} . However, the operator theory covers more general problems besides the one given in Example \ref{ex:3terms}. } 

\section{Simulations}
%The implementation details are conducted using MATLAB R2020a on a laptop equipped with an AMD Ryzen CPU operating at 3.4 GHz and 64 GB of RAM.
In this section, we evaluate the performance of our algorithm on convex quadratically constrained quadratic programs (QCQPs), see \eqref{eq:qcqp},  using synthetic and real data. {Then, we also test our algorithm on a pseudo-convex problem  using synthetic data}. We compare our Adaptive Forward-Backward-Forward  (AFBF) algorithm to Tseng's algorithm \cite{Tse:00}, and one dedicated commercial optimization software package, Gurobi \cite{Gurobi} (which has a specialized solver for QCQPs).  We implemented the algorithm AFBF as follows:  at each iteration $k\in \mathbb{N}$, the stepsize $\gamma_{k} = \bar\gamma_{k}$, where $\bar\gamma_{k}$ is computed as in \eqref{eqt:stepsize}, $b(x_{k})$ and $c(x_{k})$ are computed as in \eqref{eq:34}, and $\alpha_{k} = 0.99$. The code was implemented using MATLAB R2020a on a computer equipped with an AMD Ryzen CPU operating at 3.4 GHz and 64 GB of RAM.

\subsection{Solving convex QCQPs}
We consider the following convex QCQP 

\vspace{-0.5cm}

\begin{align}
            \min_{x\in \mathbb{R}^n} \,\ & f(x) = \frac{1}{2}x^\top Q_{0}x + b^\top x \nonumber \\
            \text{s.t.} \quad & g_{i}(x) = \frac{1}{2} x^\top Q_{i}x + l^\top_{i}x - r_{i}  \leq 0, \quad \forall i \in \{1,\ldots,m\}, \label{prob:sim}
\end{align}

\vspace{-0.8cm}

\noindent where $(Q_{i})_{0\le i \le m}$ are symmetric 
positive semidefinite  matrices in
$\mathbb{R}^{n \times n}$,  $(l_{i})_{1\le i \le m}$ and $b$ are vectors in $\mathbb{R}^{n}$, and $(r_i)_{1\le i \le m}$ are nonnegative reals.  Note that the operator $A$ defined in \eqref{opA:QCQP} for  QCQPs fits \eqref{eq:34}.  For every $i\in \{0,\ldots,m\}$,
 $Q_i$ was generated as $Q_i = R_i^\top R_i$, where $R_i \in \mathbb{R}^{p\times n}$ is a sparse random matrix whose element are drawn independently from a uniform distribution over $[0,1]$. Moreover, the components of vectors $b$ and $(l_{i})_{1\le i \le m}$
were  generated from a standard normal distribution $\mathcal{N}(0,1)$. Constants $(r_{i})_{1\le i \le m}$ and the components of the algorithm starting point were generated from a uniform distribution over $[0,1]$. For the algorithm  in \cite{Tse:00}, named \texttt{Tseng}, the line-search is computed as in \cite[equation (2.4)]{Tse:00}, with $\theta = 0.995$, $\sigma=1$, and $\beta = 0.5$. We consider the following stopping criteria for AFBF  and Tseng's algorithms: 

\vspace{-0.5cm}

\begin{equation*}
    \|u_{k}\| \leq 10^{-2},  \; \text{with}\;  u_{k}  \; \text{defined in \eqref{eq:15}}. 
\end{equation*}

\vspace{-0.6cm}

\begin{table}[h!]
	\scriptsize
 \centering 
	\begin{tabular}{| c| c| c| c | c | c | c |  c |}
		\hline
		\multirow{2}{*}{\textbf{n}} & \multirow{2}{*}{\textbf{p}}	& \multirow{2}{*}{\textbf{m}} & \multicolumn{2}{|c|}{\textbf{AFBF}}    & \multicolumn{3}{|c|}{\textbf{Tseng} \cite{Tse:00}} \\ \cline{4-8}  
        & & & ITER & CPU & ITER & LSE & CPU  \\
        \hline
        $10^3$ & $10^3$ & 250 & 3914 & \textbf{36.09}  & 15298 & 91513 & 387.4
        \\
        $10^3$ & $10^3$ & 500 & 7563 & \textbf{131.8} & 23400 & 140070  & 1179.3 \\
        $10^3$ & $10^3$ & $10^3$ & 19044 & \textbf{597.6} & 37932 & 227029 & 3570.4\\
        $10^3$ & $10^3$ & $2\cdot10^3$ & 44039 & \textbf{2900.1} & 63143 & 377963 & 12990 \\ 
        %\hline
        $10^4$ & $10^4$ & 125 & 4705 & \textbf{195.5} & 3351 &  19963 & 418.6 \\
        $10^4$ & $10^4$ & 250 & 6131 & \textbf{475.2} & 4888 & 29209 & 1178 \\
        $10^4$ & $10^4$ & 500 & 8862 & \textbf{1329} &  7240 & 43319 & 3398 \\
        $10^4$ & $10^4$ & 750 & 11380 & \textbf{1821} & 8670 & 51893 & 4251
        \\ 
        \hhline{|=|=|=|=|=|=|=|=|}
        $10^3$ & 500 & 250 & 4992 & \textbf{66.9} & 14750 & 88223 & 590.9  \\
        $10^3$ & 500 & 500 & 11069 & \textbf{288.7} & 25741 & 154114 & 2068.7 \\
        $10^3$ & 500 & $10^3$ & 24460 & \textbf{1192.7} & 45654 & 273360 & 7010.4 \\
        $10^3$ & 500 & $2\cdot10^3$ & 59762 & \textbf{5939}  &  * & * & * \\
       % \hline
        $10^4$ & $5 \cdot 10^3$ & 125 & 5318 & \textbf{336}  & 3428 & 20412 & 689.8 \\
        $10^4$ & $5 \cdot 10^3$ & 250 & 7445 & \textbf{895.3}  & 4762 & 28452 & 1864\\ 
        $10^4$ & $5 \cdot 10^3$ &  500 & 11515 & \textbf{2711}  & 11271 & 67514 & 8647  \\
        $10^4$ & $5 \cdot 10^3$ & 750 & 15719 & \textbf{3655.4}  & 14073 & 84324 & 10462 \\
        \hline
	\end{tabular}
	\begin{center}
		\caption{CPU time (sec) and number of iterations (ITER) for solving synthetic QCQPs of the form \eqref{prob:sim} with  AFBF and Tseng's \cite{Tse:00} algorithms: strongly convex case (top) and convex case (bottom).}
	\end{center}
 \vspace{-0.8cm}
 \label{table:1}
\end{table}

\noindent\noindent The CPU time (in seconds) and the number of iterations (ITER) required by each algorithm for solving problem \eqref{prob:sim} are given in Table \ref{table:1}, where ``*" means that the corresponding algorithm needs more than 5 hours to solve the problem.   Moreover, for Tseng's algorithm, we also report the number of line-search evaluations (LSE). The first half of the table corresponds to strongly convex functions ($Q_i \succ 0$, for every $i\in \{0,\ldots,m\}$) and the other half is for convex functions ($Q_i \succeq 0$, for every $i\in \{0,\ldots,m\}$). As we can notice in  Table~\ref{table:1},  AFBF outperforms Tseng's algorithm (sometimes even $10\times$ faster). Comparisons with Gurobi software are not included in Table~\ref{table:1}, since we observed that its performance is quite poor on these large test cases.

\subsection{Solving multiple kernel learning in support vector machine}
\noindent In this section, we test AFBF  on Support Vector Machine (SVM) with multiple kernel learning using real data, which can also be formulated as a convex QCQP. Let us briefly describe the problem (our presentation follows \cite{CheLi:21}). Given a set of $n_\text{dat}$ data points $\mathcal{S} = \{(d_{j},l_{j})\}_{1\le j\le n_\text{dat}}$ where, for every $j\in \{1,\ldots,n_\text{dat}\}$ $d_j \in \mathbb{R}^{n_d}$ is the input vector and $l_j \in \{-1, 1\}$ is its class label, SVM searches for a hyperplane that can best separate the points from the two classes. When the data points cannot be separated in the original space $\mathbb{R}^{n_d}$, we can search in a feature space $\mathbb{R}^{n_f}$, by mapping the input data space $\mathbb{R}^{n_d}$ to the feature space through a function $\varphi: \mathbb{R}^{n_d} \to \mathbb{R}^{n_f}$. Using function $\varphi$, we can define a kernel function $\kappa: \mathbb{R}^{n_d} \times \mathbb{R}^{n_d}\to \mathbb{R}$ as $\kappa(d_j, d_{j'}):= \langle \varphi (d_j), \varphi(d_{j'}) \rangle$ for every $(d_j, d_{j'}) \in (\mathbb{R}^{n_d})^2$, where $\langle \cdot, \cdot \rangle$ denotes the inner product of $\mathbb{R}^{n_f}$. One popular choice of the  kernel function is the Gaussian kernel: 
\begin{equation*}
     \kappa_\text{GAU} (d_j, d_{j'}) = \exp\left(-\frac{\|d_j -  d_{j'}\|^2}{2\bar\sigma^2}\right), \quad \forall (j,j') \in \{1,\ldots,n_\text{dat}\}^2
\end{equation*}

\vspace{-0.2cm}

\noindent with $\bar\sigma > 0$.  We separate the given set $\mathcal{S}$ into a training set, $\mathcal{S}_\text{tr} = \{(d_{j},l_{j})\}_{1\le j\le n_\text{tr}}$ and a testing set,  $\mathcal{S}_\text{te} = \{(d_{j},l_{j})\}_{1\le j\le n_\text{te}}$, such that $n_\text{tr} + n_\text{te} = n_\text{dat}$. Choosing a set of kernel functions $(\kappa_i)_{1\le i\le m}$, the SVM classifier is learned by solving the following convex QCQP problem on the training set  $\mathcal{S}_\text{tr}$:

\vspace{-0.9cm}

\begin{align}
    \min_{x \in \mathbb{R}^{n_{t_r}}, x_{0} \in \mathbb{R}, x \geq 0} & \;\dfrac{1}{2}  x^\top Q_0   x - e^\top x + Rx_0 \nonumber \\
  \text{s.t.} \quad &
  \dfrac{1}{2} x^\top \left(\dfrac{1}{R_i} G_{i}(K_{i,tr})\right)x - x_0 \leq 0 \quad \forall i\in \{1,\ldots,m\}, \;\;  
  \sum_{j=1}^{n_{t_r}} l_{j}x_{j} =0,
  \label{MuliSVM} 
\end{align}

\vspace{-0.9cm}

\noindent where $Q_0= C^{-1} I_{n_{t_r}}$, $C$ being a parameter related the soft margin criteria, and the vector $e$ denotes a vector of all ones. In addition, for every $i\in \{1,\ldots,m\}$, $K_{i, \text{tr}} \in \mathbb{R}^{n_\text{tr} \times n_\text{tr}}$ is a symmetric positive semidefinite matrix, whose $(j,j')$ element is defined by the kernel function: $[K_{i,\text{tr}}]_{j,j'}:= \kappa_i(d_j, d_{j'}) $. 
%for any $d_j, d_{j'} \in \mathcal{S}_\text{tr}$. 
The matrix $G_{i}(K_{i, \text{tr}}) \in \mathbb{R}^{n_\text{tr} \times n_\text{tr}}$ in the $i$-th quadratic constraint of \eqref{MuliSVM} is a symmetric positive semidefinite  matrix, its $(j,j')$ element being $[G_i(K_{i,\text{tr}})]_{j,j'} = l_j l_{j'} [K_{i,\text{tr}}]_{j,j'}$. Moreover, $R$ and $(R_i)_{1\le i \le m}$ are given positive constants.  Clearly, \eqref{MuliSVM} is an instance of problem  \eqref{eq:qcqp}. 
%Once an optimal solution $(x^*,x^*_{0})$ is found, combining it with the associated multipliers corresponding to the quadratic constraints,  $(y_i^*)_{1\le i\le m}$, and the pre-specified kernel functions $(\kappa_i)_{1\le i\le m}$ allows us to label the test data set according to the following discriminant function:
%\vspace{-0.3cm}
%\[   l =  \text{sign} \left(  \sum_{j=1}^{n_\text{tr}} l_j x_j^* \left( \sum_{i=1}^m y_i^* \kappa_i(d_j,d)  \right) + x_{0} \right), 
%\]
%\vspace{-0.2cm}
%\noindent which is applied to classify $d\in \mathcal{S}_\text{te}$. 
In our experiments,  we employed a predefined set of Gaussian kernel functions $(\kappa_i)_{1\le i\le m}$, with the corresponding $(\bar\sigma^2_i)_{1\le i \le m}$ values. Following the pre-processing strategy outlined in \cite{CheLi:21}, we normalized each matrix $K_{i, \text{tr}}$ such that $R_i= \text{trace} (K_{i,\text{tr}})$ was set to $1$, thus restricting $R = \sum_{i=1}^m R_i = m$. 
For each dataset, the $\bar\sigma^2_i$'s  were set to $m$ different grid points (generating $m$ kernels) within the interval $[10^{-1}, 10]$ for the first five datasets and $[10^{-2}, 10^2]$ for the last one, with two different values for the number of grid points, namely $m = 3$ and $m = 5$.  Additionally, we set $C = 1$. In order to give a better overview of the advantages offered by the multiple kernel SVM approach,  we also learn a  single Gaussian kernel SVM classifier with $\bar\sigma^2$  set a priori to $7$, by solving the following QP problem: 
\vspace{-0.3cm}
\begin{align}\label{s-SVM}
    \min_{ x \in [0,C]^{n_\text{tr}}} & \frac{1}{2} x^\top G(K_{\text{tr}}) x  - e^\top x, \quad \text{s.t.} \;\;\;  \sum_{j=1}^{n_\text{tr}} l_{j} x_{j} =0. 
\end{align}

\vspace{-0.9cm}

\noindent We consider the following stopping criterion for AFBF  and Tseng's algorithms: 

\vspace{-0.9cm}

\begin{align*}
  \max \left(   |f(x) - f^{*}|, \left|\sum_{j=1}^{n_\text{tr}} l_{j} x_{j}\right|, \max_{1\le i\le m} (0,g_{i}(x)) \right) \leq 10^{-4},
\end{align*}

\vspace{-0.9cm}

\noindent with $f^{*}$ computed by  Gurobi solver and the starting point chosen as the null vector. Moreover, for Tseng's algorithm the line-search was computed as in \cite[equation (2.4)]{Tse:00}, with $\theta = 0.99$, $\sigma=1$, and $\beta = 0.1$. Table \ref{table:2} presents a comparison between AFBF algorithm , Tseng's algorithm  \cite{Tse:00}, and Gurobi solver \cite{Gurobi}  in terms of CPU time for solving the QCQP of the form  \eqref{MuliSVM} using 6 real datasets \texttt{Ozone-level-8hr}, \texttt{mfeat-fourier}, \texttt{USPS}, \texttt{isolet}, \texttt{semeion}, and \texttt{Ovarian} from \url{https://www.openml.org}.  Each dataset was divided into a training set comprising $80\%$ of the data and a testing set of  the remaining $20\%$. For each dataset, we also provided the nonzero optimal dual multiplier  value $y^*$ corresponding to the unique active  quadratic inequality constraint and the corresponding value of $\bar\sigma^2$ corresponding to that active constraint. Finally, the table presents a comparison between the Testing Set Accuracies on the remaining testing datasets obtained by the multiple Gaussian kernel SVM classifier with $\bar\sigma^2$ derived from  \eqref{MuliSVM}, named TSA, and the single Gaussian kernel SVM classifier with $\bar\sigma^2=7$, named TSA0. 

\vspace{-0.2cm}

\begin{table}[h!]
\scriptsize
	\centering 
	\begin{tabular}{ | c | c| c |c | c | c| c | c| c| c| c|}
		\hline
        Dataset & \multirow{2}{*}{m}  & \multirow{2}{*}{TSA0} & \multirow{2}{*}{TSA} & \multirow{2}{*}{$\bar\sigma^2$} & \multicolumn{2}{|c|}{AFBF} & \multicolumn{2}{|c|}{TSENG} & \multicolumn{2}{|c|}{Gurobi} \\ \cline{6-11}
        ($n$, $n_{d}$) &  &  &  & & CPU &$y^{*}$ & CPU &$y^{*}$ & CPU &$y^{*}$\\  \hline
        Ozone-level-8hr & 3 & \multirow{2}{*}{52.7} & 91.7 & 5.05 & \textbf{31.18} & 3.1 & 58.09 & 2.99 & 95.61  & 3 \\  
        (2534, 72) & 5 &  & 91.7 & 2.575 & \textbf{49.9} & 5.04 & 61.38 & 5 & 339.88 & 5 \\ \hline
        mfeat-fourier & 3 &  \multirow{2}{*}{87.7} & 89 & 5.05 & \textbf{11.82} & 3.04 & 21.5 & 2.99 & 40.56 & 3 \\ 
        (2000, 76) & 5 &  & 89 & 2.575 & \textbf{20.54} & 5.02  & 35.06 & 4.99 & 170.06 & 5  \\ \hline
%        pc1 
        USPS & 3  & \multirow{2}{*}{60.2}  & 91.5 & 10 & \textbf{4} & 3  & 5.23 & 3 & 232.98 & 3 \\
       (1424, 256) & 5 & & 92.2  & 10 & \textbf{3.95} & 5 & 8.33 & 5 & 1106.7 & 5  \\ \hline
        isolet & 3 & \multirow{2}{*}{57.5}  & 95 & 10 & \textbf{0.59} & 3 & 1.35 & 3 & 10.8 & 3 \\
        (600, 617) & 5 &  & 95.8 & 10 & \textbf{0.68} & 4.97 & 2.23 & 5 & 25.09 & 5 \\ \hline
        semeion & 3 & \multirow{2}{*}{47.6} & 77.8 & 10 & \textbf{0.75} & 2.98 & 1.43 & 2.97 & 1.37 & 3 \\
        (319, 256) & 5 &  & 84.1 & 10 & \textbf{0.89} & 5.02 & 3.19  & 4.99 & 4.12 & 5 \\ \hline
        Ovarian & 3 & \multirow{2}{*}{66}  & 78 & 100 & \textbf{0.38} & 3.04  & 1.72 & 2.99 & 0.82 & 3 \\ 
        (253, 15154) & 5 &  & 88 & 100 & \textbf{0.47} & 4.96 & 2.48 & 4.99 & 2.31 & 5 \\ \hline
	\end{tabular}
  \begin{center}
		\caption{ Comparison between  our algorithm AFBF, Tseng's algorithm \cite{Tse:00} and Gurobi  solver \cite{Gurobi} in terms of CPU time (in seconds) to solve QPQCs of the form \eqref{MuliSVM} for various real datasets and two different choices of $m = 3, 5$. Additionally, TSA's are provided for \eqref{MuliSVM} and \eqref{s-SVM}.}
	\end{center}
 \label{table:2}
\end{table}

\vspace{-0.8cm}

%%%%%%%%%%%%%%%%%%%%%

\subsection{Fractional programming}
In this final set of experiments, we consider the linear fractional program \eqref{LinFracPro}, where the objective function  is pseudo-convex.  We compare our algorithm  with \cite[Algorithm 1]{ThoVuo:19}  developed for solving non-Lipschitzian and pseudo-monotone variational inequalities. We implemented \cite[Algorithm 1]{ThoVuo:19}  with the parameters   $\mu = 0.995$, $\gamma=1$, and $l = 0.001$. Note that $\mu$ in Algorithm 1 in \cite{ThoVuo:19} has a different meaning from our $\mu$ in Assumption \ref{ass1}.\ref{ass1vi}. Actually, according to Example 3.4, we have in our case $\mu=2$ for this fractional programming problem. Hence,  we use the first choice of the stepsize. From Theorem 1 in \cite{MarCar:12}, 
when the vector $r=\eta d$ with $\eta\geq0$, the objective function $f$ in \eqref{LinFracPro} is pseudo-convex on $D$.

\vspace{-0.1cm}

\begin{figure}[ht!]
	\centering
	\includegraphics[width=6cm, height=5cm]{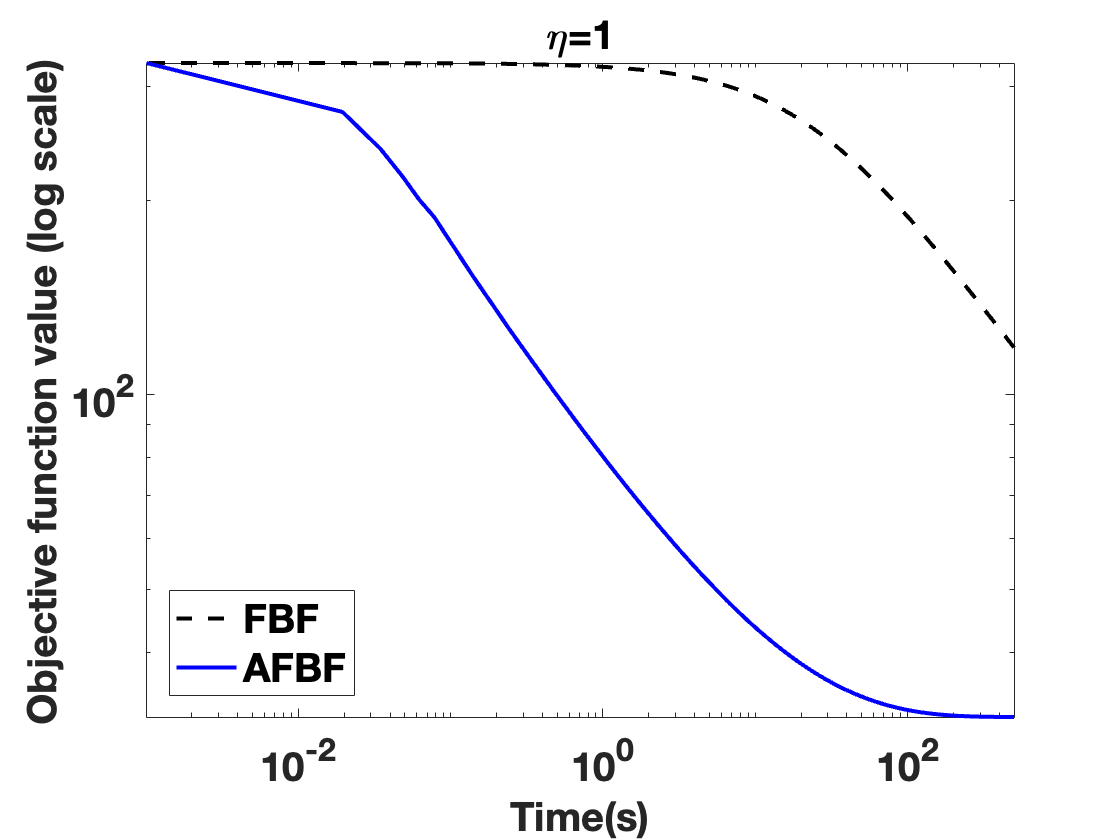} 
	\includegraphics[width=6cm,height=5cm]{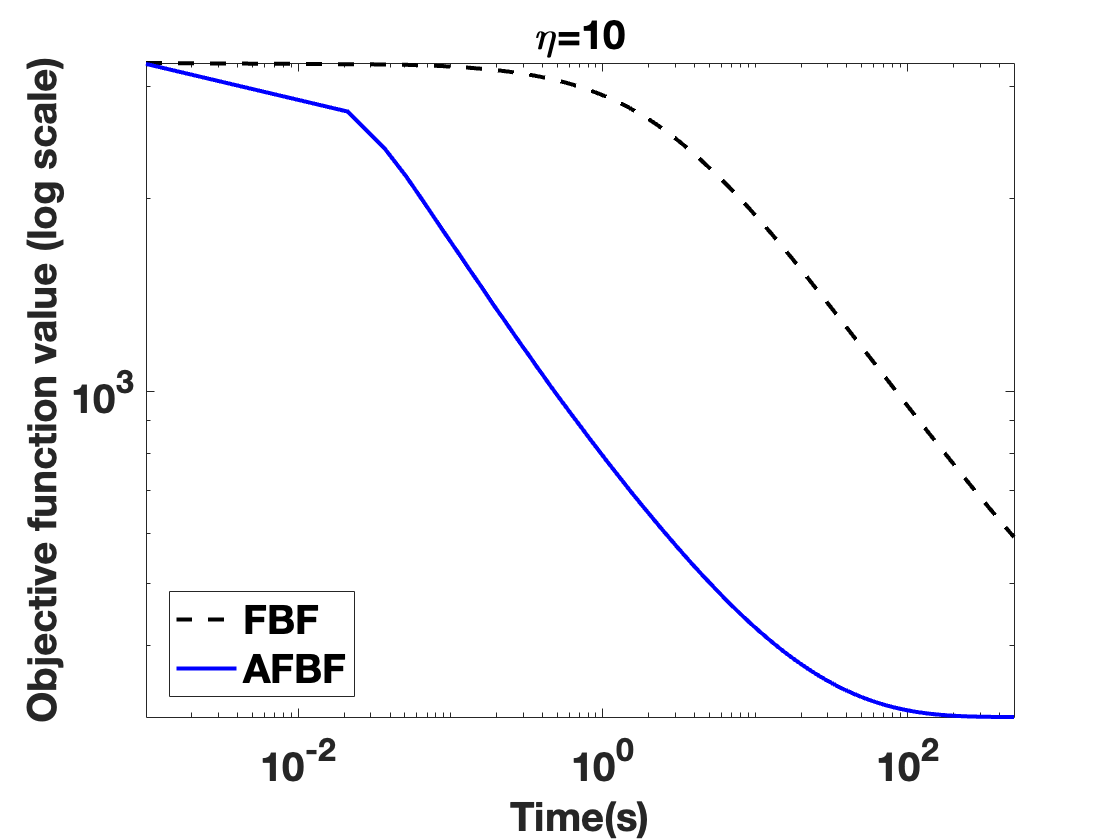} 
\vspace{-0.1cm}
	\caption{Evolution of Algorithm 1 in \cite{ThoVuo:19} (called here FBF) and our AFBF algorithm in function values along time for two linear fractional programs of the form \eqref{LinFracPro} with data generated randomly, $\eta=1$ and $\eta=10$, and  dimension  $n=10^6$.}
 \label{fig1}
\end{figure}

\noindent In our simulations, the components of the vector $d$ and  the constant $h_0$ were drawn independently from a standard normal distribution $\mathcal{N}(0,1)$, vector $r$ was chosen as $r=\eta d$, with $\eta>0$, vector $h$ was taken as a perturbation of vector $d$, i.e., $h = d + 0.01 \nu$. Vector $\nu$ and constant $d_{0}$ were generated from a uniform distribution. Moreover, we chose the starting point $x_{0}$ as $x_{0}=\operatorname{proj}_{D}(t)$, vector $t$ being generated from a standard normal distribution $\mathcal{N}(0,1)$. The results are displayed in Figure \ref{fig1}, where we plot the evolution of function values along time (in sec). Note that, AFBF is faster than Algorithm 1 from \cite{ThoVuo:19} (named here FBF) for chosen values of $\eta$.

\vspace{-0.2cm}

%\section{Results and Discussion}\label{sec:Results} ?????????

%%%%%%%%%%%%%%%%%%%%%%%

\section{Conclusion}
\label{sec:Conclusion}
In this paper, we have addressed the problem of finding a zero of 
a pseudo-monotone operator. We have made the assumption that this operator
can be split as a sum of three operators: the first continuous operator $A$ satisfies a generalized Lipschitz inequality, the second operator $B$ is Lipschitzian,  and the third one $C$ is maximally monotone. For solving this  challenging problem, our solution relied upon the forward-backward-forward algorithm, which requires however the use of an iteration-dependent stepsize. In this context, we designed two novel adaptize stepsize strategies. We also derived asymptotic sublinear convergence properties under the considered assumptions. Additionally, when $A+B+C$  satisfies a uniform pseudo-monotonicity condition, the convergence rate becomes even linear. Preliminary numerical results confirm the good performance of our algorithm. 

\medskip 

\noindent For future research, it would be intriguing to investigate the possibility of achieving more precise convergence rates. For instance, in Example \ref{ex:3terms}, when $g=0$ and $L=I_n$, \cite{Nes:15} introduces a universal gradient method with a convergence rate of order $O(\epsilon^{-2/(1+\nu)})$ for the convex (i.e., maximally monotone) case, where $\nu$ is the constant from Definition \ref{def:holder} (note that $\mu=2 \nu$ in this scenario). Conversely, in the nonconvex (i.e., nonmonotone) case under the same settings, \cite{Yas:16} examines a gradient-type method with an adaptive stepsize and achieves a convergence rate of order $O(\epsilon^{-\left(\frac{1 + \nu}{\nu}\right)})$ in the norm of the gradient. On the other hand, the convergence rate obtained in Theorem \ref{theo:convratemule2} within the general nonmonotone framework we considered is of order $O(\epsilon^{-2/\nu})$ in the norm of the gradient, which is not as favorable as the rate  in \cite{Yas:16}.

%%%%%%%%%%%%%%%%%%%%%%%

%\THEEndNotes
%\begingroup \parindent 0pt \parskip 4ex
%\def\enotesize{\normalsize} 
%\theendnotes
%\endgroup

% Acknowledgments here
%\ACKNOWLEDGMENT{We would like to express our sincere gratitude to [acknowledge individuals, organizations, or institutions] for their invaluable contributions to this research. We are also grateful to [mention any additional acknowledgements, such as technical assistance, data providers, or colleagues] for their support and assistance throughout the course of this work.}

\begin{APPENDIX}{}
\noindent 		\cite[Lemma 8.(i)]{NecCho:21} Let $\zeta > 0$ and  $(\Delta_{k})_{k\geq0}$ be a  decreasing sequence of positive numbers satisfying the following recurrence:  
		\begin{equation*}
  (\forall k \geq 0)\quad 
			\Delta_{k} - \Delta_{k+1} \geq \Delta_{k}^{\zeta+1}.
		\end{equation*}
		\noindent Then, $\Delta_k \to 0$ with sublinear rate:
			\begin{equation*}
  (\forall k \geq 0)\quad    
				\Delta_{k} \leq \dfrac{ \Delta_{0}}{\left( \zeta \Delta_{0}^{\zeta} \, k + 1\right)^\frac{1}{\zeta}} \leq \left( \dfrac{1}{\zeta k}\right) ^{\frac{1}{\zeta}}. 
			\end{equation*} 	
\end{APPENDIX}

%%%%%%%%%%%%%%%

% References here (outcomment the appropriate case)

%\bibliographystyle{plain}
%\bibliography{references}

\begin{thebibliography}{100}
	
	\bibitem{AdlBouCau:19}
	S.~Adly, L.~Bourdin, and F.~Caubet.
	\newblock On a decomposition formula for the proximal operator of the sum of
	two convex functions.
	\newblock {\em Journal of Convex Analysis}, 26(3):699--718, 2019.
	
	\bibitem{AhoAgaTho:12}
	Chris Aholt, Sameer Agarwal, and Rekha Thomas.
	\newblock A qcqp approach to triangulation.
	\newblock In Andrew Fitzgibbon, Svetlana Lazebnik, Pietro Perona, Yoichi Sato,
	and Cordelia Schmid, editors, {\em Computer Vision -- ECCV 2012}, pages
	654--667, Berlin, Heidelberg, 2012. Springer Berlin Heidelberg.
	
	\bibitem{AlaKimWri:24}
	A.~Alacaoglu, D.~Kim, and S.J. Wright.
	\newblock Extending the reach of first-order algorithms for nonconvex min-max
	problems with cohypomonotonicity, 2024.
	\newblock Preprint.
	
	\bibitem{Aus:98}
	D.~Aussel.
	\newblock Subdifferential properties of quasiconvex and pseudoconvex functions:
	unified approach.
	\newblock {\em Journal of Optimization Theory and Applications}, 97(1):29--45,
	1998.
	
	\bibitem{BauCom:11}
	H.~Bauschke and P.~Combettes.
	\newblock {\em Convex analysis and monotone operator theory in Hilbert spaces}.
	\newblock Springer, second edition, 2017.
	
	\bibitem{BorLew:06}
	J.~M. Borwein and A.~S. Lewis.
	\newblock {\em Convex analysis and nonlinear optimization: theory and
		examples}.
	\newblock Springer Science and Business Media, 2006.
	\newblock New York.
	
	\bibitem{BotCseVuo:20}
	R.~I. Bot, E.~R. Csetnek, and P.~T. Vuong.
	\newblock The forward-backward-forward method from continuous and discrete
	perspective for pseudo-monotone variational inequalities in hilbert spaces.
	\newblock {\em European Journal of Operational Research}, 287(1):49--60, 2020.
	
	\bibitem{BriDav:18}
	L.M. Briceno-Arias and D.~Davis.
	\newblock Forward-backward-half forward algorithm for solving monotone
	inclusions.
	\newblock {\em SIAM Journal on Optimization}, 28:2839--2871, 2018.
	
	\bibitem{Bro:67}
	F.~E. Browder.
	\newblock Convergence theorems for sequences of nonlinear operators in banach
	spaces.
	\newblock {\em Mathematische Zeitschrift}, 100:201--225, 1967.
	
	\bibitem{BuiCom:22}
	M.~N. B\`ui and P.~L. Combettes.
	\newblock Multivariate monotone inclusions in saddle form.
	\newblock {\em Mathematics of Operations Research}, 47:1082‚Äì1109, 2022.
	
	\bibitem{CaiZhe:22}
	Y.~Cai and W.~Zheng.
	\newblock Accelerated single-call methods for constrained min-max optimization.
	\newblock In {\em International Conference on Learning Representations}, 2022.
	
	\bibitem{CamCroMar:02}
	A.~Cambini, J.-P. Crouzeix, and Laura Martein.
	\newblock On the pseudoconvexity of a quadratic fractional function.
	\newblock {\em Optimization}, 51(4):677--687, 2002.
	
	\bibitem{CarMar:07}
	L.~Carosi and L.~Martein.
	\newblock On the pseudoconvexity and pseudolinearity of some classes of
	fractional functions.
	\newblock {\em Optimization}, 56:385--398, 2007.
	
	\bibitem{ChaPoc:11}
	A.~Chambolle and T.~Pock.
	\newblock A first-order primal-dual algorithm for convex problems with
	applications to imaging.
	\newblock {\em J. Mathematical Imaging and Vision}, 40:120--145, 2011.
	
	\bibitem{CheLi:21}
	R.~Chen and A.~L. Liu.
	\newblock A distributed algorithm for high-dimension convex quadratically
	constrained quadratic programs.
	\newblock {\em Comput. Optim. and Appl.}, 80:781--830, 2021.
	
	\bibitem{Com:24}
	P.~L. Combettes.
	\newblock The geometry of monotone operator splitting methods.
	\newblock {\em Acta Numerica}, 33:487‚Äì632, 2024.
	
	\bibitem{ComPes:12}
	P.~L. Combettes and J.-C. Pesquet.
	\newblock Primal-dual splitting algorithm for solving inclusions with mixtures
	of composite, lipschitzian, and parallel-sum type monotone operators.
	\newblock {\em Set-Valued}, 20:307--330, 2012.
	
	\bibitem{Con:13}
	L.~Condat.
	\newblock A primal-dual splitting method for convex optimization involving
	{L}ipschitzian, proximable and linear composite terms.
	\newblock {\em Journal of Optimization Theory and Applications}, 158:460--479,
	2013.
	
	\bibitem{DavWot:17}
	D.~Davis and W.~Yin.
	\newblock A three-operator splitting scheme and its optimization applications.
	\newblock {\em Set-Valued and Variational Analysis}, 25(4):829--858, 2017.
	
	\bibitem{Mai:11}
	Antonio De~Maio, Yongwei Huang, Daniel~P. Palomar, Shuzhong Zhang, and Alfonso
	Farina.
	\newblock Fractional qcqp with applications in ml steering direction estimation
	for radar detection.
	\newblock {\em IEEE Transactions on Signal Processing}, 59(1):172--185, 2011.
	
	\bibitem{DiaDasJor:21}
	J.~Diakonikolas, C.~Daskalakis, and M.~Jordan.
	\newblock Efficient methods for structured nonconvex-nonconcave min-max
	optimization.
	\newblock In {\em International Conference on Artificial Intelligence and
		Statistics}, volume 130, pages 2746--2754, 2021.
	
	\bibitem{Gurobi}
	Gurobi.
	\newblock Gurobi optimizer reference manual.
	\newblock \url{https://www.gurobi.com}.
	
	\bibitem{HasJad:07}
	A.~Hassouni and A.~Jaddar.
	\newblock On pseudoconvex functions and applications to global optimization.
	\newblock In {\em ESAIM}, volume~20, pages 138--148, 2007.
	
	\bibitem{HinSidSoh:20}
	O.~Hinder, A.~Sidford, and N.~Sohoni.
	\newblock Near-optimal methods for minimizing star-convex functions and beyond.
	\newblock In {\em Conference on Learning Theory}, volume 125, pages 1894--1938,
	2020.
	
	\bibitem{HuaPal:14}
	Y.~Huang and D.~P. Palomar.
	\newblock Randomized algorithms for optimal solutions of double-sided qcqp with
	applications in signal processing.
	\newblock {\em IEEE Transactions on Signal Processing}, 62(5):1093--1108, 2014.
	
	\bibitem{JiaVan:22}
	X.~Jiang and L.~Vandenberghe.
	\newblock Bregman three-operator splitting methods.
	\newblock {\em Journal of Optimization Theory and Applications}, 196:936--972,
	2023.
	
	\bibitem{KarSch:90}
	S.~Karamardian and S.~Schaible.
	\newblock Seven kinds of monotone maps.
	\newblock {\em Journal of Optimization Theory and Applications}, 66:37--46,
	1990.
	
	\bibitem{Lan:02}
	G.~R. Lanckriet, N.~Cristianini, P.~L. Bartlett, L.~E. Ghaoui, and M.~Jordan.
	\newblock Learning the kernel matrix with semi-definite programming.
	\newblock {\em Journal of Machine Learning Research}, 5:27--72, 2004.
	
	\bibitem{Man:65}
	O.~L. Mangasarian.
	\newblock Pseudo-convex functions.
	\newblock {\em SIAM Journal on Control and Optimization}, 3(2):281--290, 1965.
	
	\bibitem{MarCar:12}
	L.~Martein and L.~Carosi.
	\newblock The sum of a linear and a linear fractional function: pseudoconvexity
	on the nonnegative orthant and solution methods.
	\newblock {\em Bulletin of Malaysian Mathematical Sciences Society},
	35(2A):591--599, 2012.
	
	\bibitem{MerPapPil:18}
	Panayotis Mertikopoulos, Christos Papadimitriou, and Georgios Piliouras.
	\newblock Cycles in adversarial regularized learning.
	\newblock In {\em ACM-SIAM Symposium on Discrete Algorithms}, pages 2703--2717,
	2018.
	
	\bibitem{NecCho:21}
	I.~Necoara and F.~Chorobura.
	\newblock Efficiency of stochastic coordinate proximal gradient methods on
	nonseparable composite optimization.
	\newblock {\em Mathematics of Operations Research}, doi: 10.1287/moor.2023.0044, 2024.
	
	\bibitem{Nem:04}
	A.~Nemirovski.
	\newblock Prox-method with rate of convergence $\mathcal{O}(1/t)$ for
	variational inequalities with {L}ipschitz continuous monotone operators and
	smooth convex-concave saddle point problems.
	\newblock {\em SIAM Jounal on Optimization}, 15:229--251, 2004.
	
	\bibitem{Nes:15}
	Y.~Nesterov.
	\newblock Universal gradient methods for convex optimization problems.
	\newblock {\em Mathematical Programming}, 152(1-2):381--404, 2015.
	
	\bibitem{Pet:22}
	T.~Pethick, P.~Latafat, P.~Patrinos, O.~Fercoq, and V.~Cevher.
	\newblock Escaping limit cycles: global convergence for constrained
	nonconvex-nonconcave minimax problems.
	\newblock In {\em International Conference on Learning Representations}, 2022.
	
	\bibitem{Sch:81}
	S.~Schaible.
	\newblock Quasiconvex, pseudoconvex, and strictly pseudoconvex quadratic
	functions.
	\newblock {\em Journal of Optimization Theory and Application}, 35(3):303--338,
	1981.
	
	\bibitem{ThoVuo:19}
	D.~V. Thong and P.~T. Vuong.
	\newblock Modified tseng's extragradient methods for solving pseudo-monotone
	variational inequalities.
	\newblock {\em Optimization}, 68:2207--2226, 2019.
	
	\bibitem{Ton:23}
	B.~Tongnoi.
	\newblock A modified tseng‚Äôs algorithm with extrapolation from the past for
	pseudo-monotone variational inequalities.
	\newblock {\em Taiwanese J. of Mathematics}, 28(1):187--210, 2024.
	
	\bibitem{Tse:00}
	P.~Tseng.
	\newblock A modified forward-backward splitting method for maximal monotone
	mappings.
	\newblock {\em SIAM Journal on Control and Optimization}, 38(2):431--446, 2000.
	
	\bibitem{Yas:16}
	M.~Yashtini.
	\newblock On the global convergence rate of the gradient descent method for
	functions with Holder continuous gradients.
	\newblock {\em Optimization Letters}, 10(6):1361--1370, 2016.
	
\end{thebibliography}

%%%%%%%%%%%%%%%%%
\end{document}